\documentclass[11pt, a4paper]{article}

\usepackage[whole]{bxcjkjatype}

\usepackage{bm, amsmath, amsthm, amssymb, accents, comment}
\usepackage{ascmac}
\RequirePackage{amsthm,amsmath,amsfonts,amssymb}
\RequirePackage{natbib}
\RequirePackage{graphicx}\usepackage{tikz}
\usepackage{pgfplots}
\usepackage{subcaption}
\pgfplotsset{
	compat=newest, 
	cycle list name=exotic }

\newtheorem{theorem}{Theorem}[section]
\newtheorem{proposition}{Proposition}[section]
\newtheorem{lemma}{Lemma}[section]
\newtheorem{corollary}{Corollary}[section]

\usepackage{typearea}
\typearea{12}

\begin{document}
	
	\title{Inadmissibility of the corrected Akaike information criterion}
	
	\author{Takeru Matsuda\thanks{Department of Mathematical Informatics, Graduate School of Information Science and Technology, The University of Tokyo \& Statistical Mathematics Unit, RIKEN Center for Brain Science, e-mail: \texttt{matsuda@mist.i.u-tokyo.ac.jp}}}
	
	\date{}
	
	\maketitle
	
	\begin{abstract}
		For the multivariate linear regression model with unknown covariance, the corrected Akaike information criterion is the minimum variance unbiased estimator of the expected Kullback--Leibler discrepancy.
		In this study, based on the loss estimation framework, we show its inadmissibility as an estimator of the Kullback--Leibler discrepancy itself, instead of the expected Kullback--Leibler discrepancy.
		We provide improved estimators of the Kullback--Leibler discrepancy that work well in reduced-rank situations and examine their performance numerically.
	\end{abstract}

\section{Introduction}
We consider the multivariate linear regression model with $p$ explanatory variables and $q$ response variables:
\begin{align}
	y_i = B^{\top} x_i + \varepsilon_i, \quad \varepsilon_i \sim {\rm N}_q (0,\Sigma), \label{fullmodel}
\end{align}
for $i=1,\dots,n$, where $n \geq p$, $B \in \mathbb{R}^{p \times q}$ is an unknown regression coefficient matrix, $\Sigma  \in \mathbb{R}^{q \times q}$ is an unknown covariance matrix (positive definite) and $\varepsilon_1,\dots,\varepsilon_n$ are independent.
In the following, the probability density function of $Y=(y_1, y_2, \dots, y_n)^{\top} \in \mathbb{R}^{n \times q}$ is denoted by $p(Y \mid B, \Sigma)$ and the expectation of $f(Y)$ under $p(Y \mid B, \Sigma)$ is written as ${\rm E}_{B,\Sigma}[ f(Y) ]$.

The maximum likelihood estimate for the model \eqref{fullmodel} is given by
\begin{align*}
	\hat{B}=(X^{\top} X)^{-1} X^{\top} Y, \quad \hat{\Sigma}=\frac{1}{n} (Y-X \hat{B})^{\top} (Y-X \hat{B}).
\end{align*}
The Akaike Information Criterion \citep[AIC;][]{Akaike}  is an approximately unbiased estimator of the expected Kullback--Leibler discrepancy:
\begin{align*}
	{\rm E}_{B,\Sigma} [ {\rm AIC} ] = {\rm E}_{B,\Sigma} [ d ( (B,\Sigma), (\hat{B} ,\hat{\Sigma}) ) ] + o(1)
\end{align*}
as $n \to \infty$, where 
\begin{align}
	d ( (B,\Sigma), (\hat{B},\hat{\Sigma}) ) &= -2 \int p(\widetilde{Y} \mid B, \Sigma) \log p (\widetilde{Y} \mid \hat{B}, \hat{\Sigma}) {\rm d} \widetilde{Y} \label{Delta} \\
	& = n q \log (2 \pi) + n \log \det \hat{\Sigma} + n {\rm tr} (\hat{\Sigma}^{-1} \Sigma) + {\rm tr} (\hat{\Sigma}^{-1} (\hat{B}-B)^{\top} X^{\top} X (\hat{B}-B)) \nonumber
\end{align}
is called the Kullback--Leibler discrepancy from $p(\widetilde{Y} \mid B,\Sigma)$ to $p( \widetilde{Y} \mid \hat{B},\hat{\Sigma})$.
The AIC is widely used for evaluation and selection of linear regression models \citep{Burnham,Konishi08}.

The bias of the AIC is non-negligible when the sample size $n$ is not sufficiently large. 
Thus, a corrected AIC (AICc) has been derived \citep{Sugiura,Hurvich,Bedrick},
which is exactly unbiased:
\begin{align*}
	{\rm E}_{B,\Sigma} [ {\rm AICc} ] = {\rm E}_{B,\Sigma} [ d ( (B,\Sigma), (\hat{B} ,\hat{\Sigma}) ) ].
\end{align*}
\cite{Cavanaugh} provided a unified derivation of AIC and AICc and \cite{Davies} showed that AICc is the minimum variance unbiased estimator of the expected Kullback--Leibler discrepancy. 

Both AIC and AICc were developed to unbiasedly estimate the expected Kullback--Leibler discrepancy.
This idea dates back to Stein's unbiased risk estimate \citep[SURE;][]{Stein74}, which unbiasedly estimates the quadratic risk of estimators of a normal mean \citep{Lehmann,shr_book}.
In this context, \cite{Johnstone} considered estimation of the quadratic loss itself, instead of its average (quadratic risk).
Although SURE is still unbiased for this problem, \cite{Johnstone} showed that it can be improved in terms of the mean squared error.
In other words, SURE is inadmissible as an estimator of the quadratic loss.
See Section~\ref{sec:loss} for details.
This finding led to the development of a field called loss estimation \citep{Fourdrinier}.

	In this study, we examine AIC and AICc from the loss estimation viewpoint and investigate their admissibility as estimators of the Kullback--Leibler discrepancy $d((B, \Sigma), (\hat{B}, \hat{\Sigma}))$, instead of its average ${\rm E}_{B,\Sigma} [d((B, \Sigma), (\hat{B}, \hat{\Sigma}))]$ (expected Kullback--Leibler discrepancy).
	The former estimand is random (depends on $Y$ as well as $(B,\Sigma)$) whereas the latter one is non-random (depends only on $(B,\Sigma)$). 
	In this sense, it may be more correct to refer to the current problem as prediction rather than estimation\footnote{Similarly, \cite{Lehmann} states that it is common to speak of prediction, rather  than estimation, of random effects (Example~3.5.5). See also \cite{Sandved}}.
	The current setting to estimate (or predict) $d((B, \Sigma), (\hat{B}, \hat{\Sigma}))$ is considered to reflect the practical usage of AIC (and AICc) as a model evaluation criterion more faithfully as follows.
	Given data at hand, we estimate $(B,\Sigma)$ by $(\hat{B},\hat{\Sigma})$ and construct the plug-in predictive distribution $p (\widetilde{Y} \mid \hat{B}, \hat{\Sigma})$ for a future observation.
	The disparity between this predictive distribution and the true data-generating distribution $p (\widetilde{Y} \mid {B}, {\Sigma})$ is given by the Kullback--Leibler discrepancy $d ( (B,\Sigma), (\hat{B},\hat{\Sigma}) )$.
	Whereas the usual argument on AIC (and AICc) considers estimation of the average of $d ( (B,\Sigma), (\hat{B},\hat{\Sigma}) )$ over the possible realizations of $Y$ (expected Kullback--Leibler discrepancy), here we focus on estimation (or prediction) of $d ( (B,\Sigma), (\hat{B},\hat{\Sigma}) )$ itself for the specific realization of $Y$ at hand.
	Thus, the current setting provides direct (conditional) assessment of the performance of the predictive distribution obtained from the data at hand.
	Note that \cite{Matsuda16} studied the Pitman closeness property of predictive distributions in a similar spirit.
	We develop improved estimators of $d ( (B,\Sigma), (\hat{B},\hat{\Sigma}) )$ and show that they attain better variable selection result than AIC and AICc in simulation.
	It demonstrates a practical advantage of introducing the current setting.
	See \cite{Fourdrinier} for further discussion on motivation for considering loss estimation.

This paper is organized as follows.
In Section~\ref{sec:loss}, we briefly review the loss estimation framework and existing results for normal mean vector and matrix.
We also derive an improved loss estimator for a normal mean matrix, which will be the basis of the main results of this paper.
In Section~\ref{sec:ic}, we introduce the general setting of loss estimation for a predictive distribution and study the properties of AIC and AICc as loss estimators in multivariate linear regression.
For the multivariate linear regression model \eqref{fullmodel} with known covariance, AIC is shown to be inadmissible and an improved loss estimator is given.
For the multivariate linear regression model \eqref{fullmodel} with unknown covariance, AIC is shown to be inadmissible and dominated by AICc.
Then, in Section~\ref{sec:main}, we prove that AICc is still inadmissible and provide improved loss estimators that work well in reduced-rank situations.
In Section~\ref{sec:experiments}, we present numerical results to examine the performance of the improved estimators.
The results demonstrate that the improved estimators often outperform the corrected AIC in variable selection.
Finally, we provide concluding remarks in Section~\ref{sec:concl}.
Technical lemmas are given with proofs in the Appendix.

\section{Loss estimation framework}\label{sec:loss}

\subsection{General setting}
Here, we briefly introduce the loss estimation framework.
See \cite{Fourdrinier,shr_book} for a comprehensive review of loss estimation.
Recently, the idea of loss estimation has been applied to high-dimensional inference \citep{Bellec}.

Suppose that we have an observation $Y \sim p(y \mid \theta)$, where $\theta$ is an unknown parameter.
In usual setting of point estimation \citep{Lehmann}, we consider estimation of $\theta$ using an estimator $\hat{\theta}=\hat{\theta}(y)$.
The discrepancy of an estimate $\hat{\theta}$ from the true value $\theta$ is quantified by a loss function $L(\theta,\hat{\theta})$.
Then, estimators are compared by using the risk function $R(\theta,\hat{\theta})={\rm E}_{\theta} [L(\theta,\hat{\theta}(y))]$, which is the average of the loss. 
An estimator $\hat{\theta}_1$ is said to dominate another estimator $\hat{\theta}_2$ if $R(\theta,\hat{\theta}_1) \leq R(\theta,\hat{\theta}_2)$ holds for every $\theta$, with strict inequality for at least one value of $\theta$.
An estimator $\hat{\theta}$ is said to be admissible if no estimator dominates $\hat{\theta}$. 
An estimator $\hat{\theta}$ is said to be inadmissible if it is not admissible (i.e. there exists an estimator that dominates $\hat{\theta}$).

In the setting above, loss estimation concerns estimation of the loss $L(\theta,\hat{\theta}(y))$, which depends not only on $\theta$ but also on $y$.
The performance of a loss estimator $\lambda(y)$ is evaluated by squared error $(\lambda(y)-L ( \theta,\hat{\theta}(y) ) )^2$.
Thus, a loss estimator $\lambda_1(y)$ is said to dominate another loss estimator $\lambda_2(y)$ if
\begin{align*}
	{\rm E}_{\theta} [(\lambda_1(y)-L ( \theta,\hat{\theta}(y) ) )^2] \leq {\rm E}_{\theta} [(\lambda_2(y)-L ( \theta,\hat{\theta}(y) ) )^2]
\end{align*}
holds for every $\theta$, with strict inequality for at least one value of $\theta$.
A loss estimator $\lambda(y)$ is said to be admissible if no loss estimator dominates $\lambda(y)$. 
A loss estimator $\lambda(y)$ is said to be inadmissible if it is not admissible (i.e. there exists a loss estimator that dominates $\lambda(y)$).

\subsection{Loss estimation for a normal mean vector}
Now, we focus on loss estimation for a normal mean vector \citep{Johnstone}.
Suppose that we estimate $\theta \in \mathbb{R}^p$ from an observation $Y \sim {\rm N}_p (\theta,I_p)$ by an estimator $\hat{\theta}(y)=y+g(y)$ under the quadratic loss 
\begin{align*}
	L ( \theta,\hat{\theta} )=\| \hat{\theta}-\theta \|^2.
\end{align*}
\cite{Stein74} showed that the quadratic risk $R(\theta,\hat{\theta})={\rm E}_{\theta} [ L ( \theta,\hat{\theta}(y) ) ]$ satisfies 
\begin{align*}
	R(\theta,\hat{\theta})={\rm E}_{\theta} [ \lambda^{{\rm U}}(y) ], 
\end{align*}
where 
\begin{align*}
	\lambda^{{\rm U}}(y)=p+2 \nabla \cdot g(y) + \| g(y) \|^2
\end{align*}
is called Stein's unbiased risk estimate (SURE).
SURE plays a central role in the theory of shrinkage estimation \citep{shr_book} and also closely related to model selection criteria such as Mallows' $C_p$ and AIC \citep{Boisbunon}.
However, \cite{Johnstone} showed that SURE is inadmissible for the maximum likelihood estimator when $p \geq 5$ as follows.

\begin{proposition}\citep{Johnstone}\label{prop_johnstone}
	In estimation of $\theta$ from $Y \sim {\rm N}_p (\theta, I_p)$ under the quadratic loss, consider the maximum likelihood estimator $\hat{\theta}(y) = y$.
	If $p \geq 5$, then SURE $\lambda^{{\rm U}}(y)=p$ is inadmissible and dominated by the loss estimator 
	\begin{align}
		\lambda(y) = p-2(p-4) \| y \|^{-2}. \label{vec_est}
	\end{align}
\end{proposition}

Figure~\ref{fig_johnstone} plots the percentage improvements in mean sqaured error of the loss estimator \eqref{vec_est} over SURE defined by
\begin{align*}
	100 \frac{{\rm E}_{\theta} [ ( \lambda^{{\rm U}}(y) - \| \hat{\theta}(y) - \theta \|^2 )^2 ]-{\rm E}_{\theta} [ ( \lambda(y) - \| \hat{\theta}(y) - \theta \|^2 )^2 ]}{{\rm E}_{\theta} [ ( \lambda^{{\rm U}}(y) - \| \hat{\theta}(y) - \theta \|^2 )^2 ]}.
\end{align*}
The improvement is large when the true value of $\theta$ is close to the origin, which is qualitatively similar to the risk behavior of the James--Stein estimator.
In addition to the maximum likelihood estimator, \cite{Johnstone} also proved the inadmissibility of SURE for the James--Stein estimator and provided improved loss estimators.
Based on these findings by \cite{Johnstone}, many studies have investigated loss estimation for a normal mean vector and single-response linear regression, such as \citep{Boisbunon,Fourdrinier03,Fourdrinier,Lu,Narayanan,Wan}.

\begin{figure}
	\centering
	\includegraphics{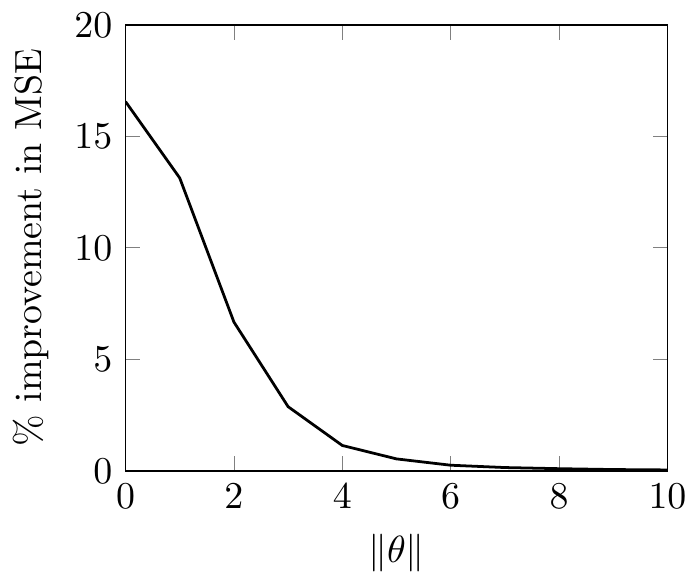}
	\caption{Percentage improvements in mean squared error of the loss estimator \eqref{vec_est} over SURE for $p=8$.}
	\label{fig_johnstone}
\end{figure}

\subsection{Loss estimation for a normal mean matrix}
Recently, \cite{Matsuda} generalized the results of \cite{Johnstone} to matrices and developed loss estimators that dominate SURE.
This is motivated from the Efron–Morris estimator, which is a matrix generalization of the James--Stein estimator that shrinks singular values towards zero \citep{Efron72}.
Specifically, suppose that we estimate $M \in \mathbb{R}^{p \times q}$ from an observation $Y \sim {\rm N}_{p,q} (M,I_p,I_q)$ by an estimator $\hat{M}(Y)$ under the Frobenius loss 
\begin{align*}
	L ( M,\hat{M} )=\| \hat{M}-M \|_{\mathrm{F}}^2 = \sum_{i,j} (\hat{M}_{ij}-M_{ij})^2.
\end{align*}
We write the singular values of a matrix $Z \in \mathbb{R}^{p \times q}$ with $p \geq q$ by $\sigma_1(Z) \geq \dots \geq \sigma_q(Z) \geq 0$.

\begin{proposition}\citep{Matsuda}\label{prop_ms}
	In estimation of $M$ from $Y \sim {\rm N}_{p,q} (M, I_p, I_q)$ under the Frobenius loss, consider the maximum likelihood estimator $\hat{M}(Y) = Y$.
	If $p \geq 3q+2$ and $q \geq 2$, then SURE $\lambda^{{\rm U}}(Y)=pq$ is inadmissible and dominated by the loss estimator 
	\begin{align}
		\lambda(Y) = pq-\sum_{i=1}^q c_i \sigma_i(Y)^{-2}, \quad c_i = \frac{4(p-q-2i-1)}{q}. \label{mat_est}
	\end{align}
\end{proposition}

Whereas Proposition~\ref{prop_ms} shows the inadmissibility of SURE, it excludes\footnote{The condition $q \geq 2$ in Proposition~\ref{prop_ms} was not explicitly stated in the original paper \citep{Matsuda}.} the case $q=1$.
Here, we provide another loss estimator dominating SURE, which reduces to \eqref{vec_est} in Proposition~\ref{prop_johnstone} when $q=1$ and will be the basis of the main results of this paper.

\begin{theorem}\label{th_ms2}
	In estimation of $M$ from $Y \sim {\rm N}_{p,q} (M, I_p, I_q)$ under the Frobenius loss, consider the maximum likelihood estimator $\hat{M}(Y) = Y$.
	If $p \geq 2q+3$, then SURE $\lambda^{{\rm U}}(Y)=pq$ is inadmissible and dominated by the loss estimator 
	\begin{align}
		\lambda(Y) = pq - \frac{2(p-2q-2)}{q} {\rm tr} ((Y^{\top} Y)^{-1}). \label{mat_est2}
	\end{align}
\end{theorem}
\begin{proof}
	Let $h(Y) = -c {\rm tr} ((Y^{\top} Y)^{-1})$ with $c=2(p-2q-2)/q$ so that $\lambda(Y) = \lambda^{\mathrm{U}}(Y) + h(Y)$.
	Then, from Lemma 5 of \cite{Matsuda},
	\begin{align*}
		{\rm E}_{M} [ ( \lambda(Y)- \| \hat{M}(Y)-M \|^2 )^2] - {\rm E}_{M} [ (\lambda^{\mathrm{U}}(Y)-\| \hat{M}(Y)-M \|^2)^2 ] = {\rm E}_{M} [ - 2 \Delta h (Y)  + h(Y)^2 ],
	\end{align*}
	where
	\begin{align*}
		\Delta h (Y) &= \sum_{i,j} \frac{\partial^2 h}{\partial Y_{ij}^2} (Y).
	\end{align*}
	From Lemma~\ref{lem_diff2} and Lemma~\ref{lem_diff3},
	\begin{align*}
		\Delta h (Y) &= -c \sum_{i,j} \frac{\partial^2}{\partial Y_{ij}^2} {\rm tr} ((Y^{\top} Y)^{-1}) \\
		&= 2c \sum_{i,j} \frac{\partial}{\partial Y_{ij}} (Y(Y^{\top} Y)^{-2})_{ij} \\
		&= 2c (p-q-2) {\rm tr} ((Y^{\top} Y)^{-2}) - 2c ({\rm tr} ((Y^{\top} Y)^{-1}))^2.
	\end{align*}
	Thus,
	\begin{align*}
		- 2 \Delta h (Y) + h(Y)^2 &= -4c (p-q-2) {\rm tr} ((Y^{\top} Y)^{-2}) + (c^2+4c) ({\rm tr} ((Y^{\top} Y)^{-1}))^2.
	\end{align*}
	From the Cauchy--Schwarz inequality,
	\begin{align}
		({\rm tr} ((Y^{\top} Y)^{-1}) )^2 = \left( \sum_{i=1}^q \lambda_i((Y^{\top} Y)^{-1}) \right)^2 \leq q \sum_{i=1}^q \lambda_i((Y^{\top} Y)^{-1})^2 = q {\rm tr} ((Y^{\top} Y)^{-2}), \label{cs}
	\end{align}
	where $\lambda_i(A)$ denotes the $i$-th eigenvalue of a matrix $A$.
	Therefore, by substituting $c=2(p-2q-2)/q$,
	\begin{align*}
		- 2 \Delta h (Y) + h(Y)^2 &\leq c(-4 (p-q-2) + q(c+4)) {\rm tr} ((Y^{\top} Y)^{-2}) \\
		&= -\frac{4(p-2q-2)^2}{q} {\rm tr} ((Y^{\top} Y)^{-2}) \\
		&< 0.
	\end{align*}
	Hence,
	\begin{align*}
		{\rm E}_{M} [ ( \lambda(Y)- \| \hat{M}(Y)-M \|^2 )^2] < {\rm E}_{M} [ (\lambda^{\mathrm{U}}(Y)-\| \hat{M}(Y)-M \|^2)^2 ]
	\end{align*}
	for every $M$.
\end{proof}

Figure~\ref{fig_matsuda} plots the percentage improvements in mean squared error of the loss estimator \eqref{mat_est2} over SURE like Figure~\ref{fig_johnstone}.
The improvement is large when some of the singular values of $M$ are small. 
In particular, the left panel of Figure~\ref{fig_matsuda} indicates that the loss estimator \eqref{mat_est2} attains constant reduction of MSE as long as $\sigma_2(M)=0$, even when $\sigma_1(M)$ is large.
Thus, the loss estimator \eqref{mat_est2} works well when $M$ is close to low-rank. 
Note that the loss estimator \eqref{mat_est} has qualitatively the same property \citep{Matsuda}.
These results are understood from the fact that both loss estimators \eqref{mat_est} and \eqref{mat_est2} are based on the inverse square of the singular values of $Y$.
The Efron–Morris estimator for a normal mean matrix has a similar risk property \citep{Matsuda22}.

\begin{figure}
	\centering
	\includegraphics{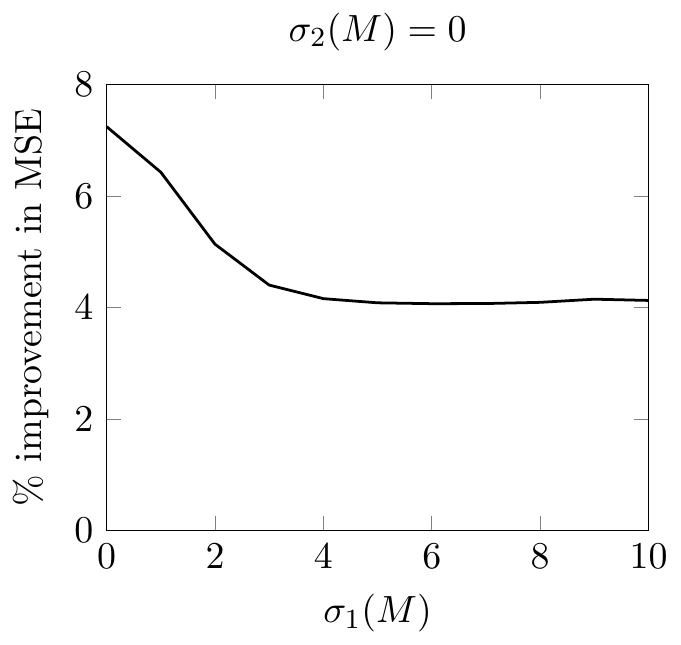}
	\includegraphics{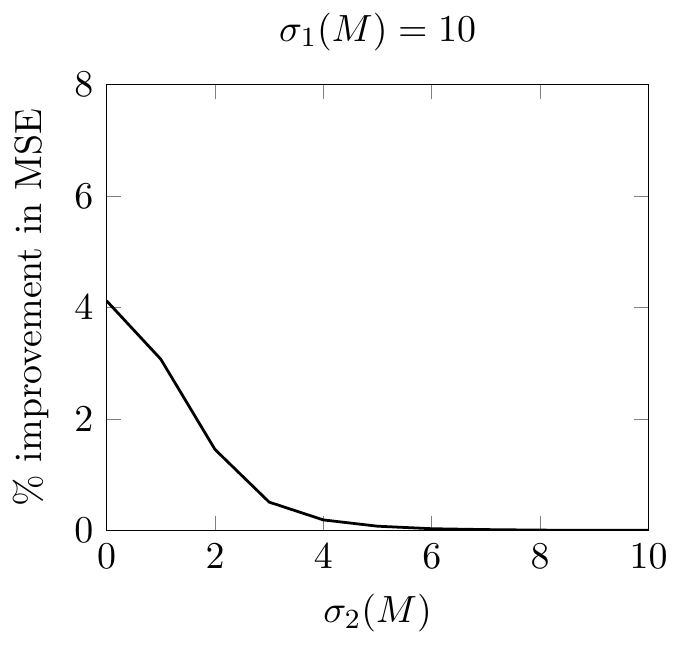}
	\caption{Percentage improvements in mean squared error of $\lambda(Y)$ in \eqref{mat_est2} over SURE for $p=8$ and $q=2$. Left: $\sigma_2(M)=0$. Right: $\sigma_1(M)=10$.}
	\label{fig_matsuda}
\end{figure}

In addition to the maximum likelihood estimator, \cite{Matsuda} also proved the inadmissibility of SURE for a general class of orthogonally invariant estimators, including the Efron--Morris estimator and reduced-rank estimators, and provided improved loss estimators.

\section{Information criterion as loss estimator}\label{sec:ic}

\subsection{Loss estimation for a predictive distribution}
Suppose that we have an observation $Y \sim p(y \mid \theta)$, where $\theta$ is an unknown parameter.
Then, we consider prediction of a future observation $\widetilde{Y} \sim p(\widetilde{y} \mid \theta)$ by using a predictive distribution $\hat{p}(\widetilde{y} \mid y)$.
The discrepancy of a predictive distribution $\hat{p}(\widetilde{y} \mid y)$ from the true distribution $p(\widetilde{y} \mid \theta)$ is evaluated by the Kullback--Leibler discrepancy
\begin{align*}
	d ( p(\widetilde{y} \mid \theta), \hat{p}(\widetilde{y} \mid y) ) &= -2 \int p(\widetilde{y} \mid \theta) \log {\hat{p} (\widetilde{y} \mid y)} {\rm d} \widetilde{y},
\end{align*}
which is equivalent to twice the Kullback--Leibler divergence 
\begin{align*}
	D ( p(\widetilde{y} \mid \theta), \hat{p}(\widetilde{y} \mid y) ) &= \int p(\widetilde{y} \mid \theta) \log \frac{p(\widetilde{y} \mid \theta)}{\hat{p} (\widetilde{y} \mid y)} {\rm d} \widetilde{y}
\end{align*}
up to an additive constant. 
The plug-in predictive distribution is defined by $p(\widetilde{y} \mid \hat{\theta}(y))$, where $\hat{\theta}(y)$ is the maximum likelihood estimate of $\theta$ from $y$.
AIC \citep{Akaike} is an approximately unbiased estimator of the Kullback--Leibler discrepancy for the plug-in predictive distribution:
\begin{align*}
	{\rm E}_{\theta} [\mathrm{AIC}] \approx d ( p(\widetilde{y} \mid \theta), {p}(\widetilde{y} \mid \hat{\theta}(y)) ).
\end{align*}
See \cite{Burnham,Konishi08} for details.

Similarly to point estimation in Section~\ref{sec:loss}, we can formulate estimation of the Kullback--Leibler discrepancy as a loss estimation problem. 
Then, AIC can be viewed as a default loss estimator like SURE in estimation of a normal mean.
From this viewpoint, it is of interest to determine whether AIC is admissible or not.
In the following, we investigate this problem for the multivariate linear regression model \eqref{fullmodel}.

\subsection{Multivariate linear regression with known covariance}
First, consider the multivariate linear regression model \eqref{fullmodel} with known covariance $\Sigma \succ O$.
The maximum likelihood estimate is $\hat{B}=(X^{\top} X)^{-1} X^{\top} Y$.
From \eqref{Delta}, the Kullback--Leibler discrepancy for the plug-in predictive distribution is
\begin{align*}
	d( (B,\Sigma), (\hat{B},\Sigma) ) = n q \log (2 \pi e) + n \log \det {\Sigma} +  {\rm tr} ({\Sigma}^{-1} (\hat{B}-B)^{\top} X^{\top} X (\hat{B}-B)).
\end{align*}
The AIC is
\begin{align*}
	{\rm AIC} &= -2 \log p(Y \mid \hat{B}, \Sigma) + 2 pq \\
	&= nq \log(2 \pi) + n \log \det {\Sigma} +  {\rm tr} ( \Sigma^{-1} (Y-X\hat{B})^{\top} (Y-X\hat{B})) + 2 pq.
\end{align*}
Then, the inadmissibility of AIC is proved as follows, where MAIC is an abbreviation of ``Modified AIC."

\begin{theorem}\label{th_aic}
	Consider the multivariate linear regression model \eqref{fullmodel} with known $\Sigma \succ O$.
	If $p \geq 2q+3$, then AIC is inadmissible and dominated by
	\begin{align*}
		{\rm MAIC} = {\rm AIC}-\frac{2(p-2q-2)}{q} {\rm tr} (\Sigma ((X \hat{B})^{\top} (X \hat{B}))^{-1})
	\end{align*}
	as an estimator of the Kullback--Leibler discrepancy.
\end{theorem}
\begin{proof}
	Let $R=(Y-X\hat{B})^{\top}(Y-X\hat{B})$ be the residual.
	Then, from the standard theory of multivariate linear regression \citep{Anderson}, $\hat{B}$ and $R$ are independent and distributed as $\hat{B} \sim {\rm N}_{p,q}(B,(X^{\top} X)^{-1},\Sigma)$ and $R \sim W_q ( n-p, \Sigma )$, respectively.
	Thus, $Z = (X^{\top} X)^{1/2} \hat{B} \Sigma^{-1/2}$ is independent from $R$ and distributed as $Z \sim {\rm N}_{p,q}(\bar{Z},I_p,I_q)$ where $\bar{Z}=(X^{\top} X)^{1/2} {B} \Sigma^{-1/2}$.
	
	Let $h = -c {\rm tr} (\Sigma ((X \hat{B})^{\top} (X \hat{B}))^{-1})=-c {\rm tr} ((Z^{\top} Z)^{-1})$ with $c=2(p-2q-2)/q$ so that ${\rm MAIC} = {\rm AIC}+h(Y)$.
	Then,
	\begin{align*}
		{\rm E}_B [ ( {\rm MAIC}-d )^2] - {\rm E}_B [ ( {\rm AIC}-d )^2] &= {\rm E}_B [ h^2 + 2h( {\rm AIC}-d )],
	\end{align*}
	where we write $d( (B,\Sigma), (\hat{B},\Sigma) )$ by $d$ for simplicity.
	Note that
	\begin{align*}
		{\rm AIC}-d &= (2p-n) q + {\rm tr} (\Sigma^{-1} (Y-X\hat{B})^{\top} (Y-X\hat{B})) - {\rm tr} (\Sigma^{-1} (\hat{B}-B)^{\top} X^{\top} X (\hat{B}-B)) \\
		&= (2p-n) q + {\rm tr} (\Sigma^{-1} R) - {\rm tr} ((Z-\bar{Z})^{\top}(Z-\bar{Z})).
	\end{align*}
	From ${\rm E}[R] = (n-p)\Sigma$ and the independence of $Z$ and $R$,
	\begin{align*}
		{\rm E}_B [{\rm tr} ((Z^{\top} Z)^{-1}) {\rm tr} (\Sigma^{-1} R)] = (n-p)q \cdot {\rm E}_B [{\rm tr} ((Z^{\top} Z)^{-1})].
	\end{align*}
	Also, from Lemma~\ref{lem_exp0},
	\begin{align*}
		&{\rm E} [ {\rm tr} ((Z^{\top} Z)^{-1}) {\rm tr} ((Z-\bar{Z})^{\top} (Z-\bar{Z}))] \\
		=& pq {\rm E} [ {\rm tr} (  (Z^{\top} Z)^{-1}  ) ] - 2 (p-q-2) {\rm E} [ {\rm tr} ( (Z^{\top} Z)^{-2} ) ] +2 {\rm E} [ ({\rm tr} ((Z^{\top} Z)^{-1}) )^2 ].
	\end{align*}
	Therefore,
	\begin{align*}
		{\rm E}_B [ ( {\rm MAIC}-d )^2] - {\rm E}_B [ ( {\rm AIC}-d )^2] &= {\rm E}_B [ (c^2+4c) ({\rm tr} ((Z^{\top} Z)^{-1}) )^2 -4c(p-q-2) {\rm tr} ((Z^{\top} Z)^{-2})] \\
		& \leq qc \left( c-\frac{4(p-2q-2)}{q} \right) {\rm E}_B [ {\rm tr} ((Z^{\top} Z)^{-2})] \\
		& < 0
	\end{align*}
	for every $B$, where we used ${({\rm tr} ((Z^{\top} Z)^{-1}) )^2} \leq q {{\rm tr} ( (Z^{\top} Z)^{-2})}$ from \eqref{cs} and $c=2(p-2q-2)/q$. 
\end{proof}

For the Gaussian linear regression model with known variance ($q=1$), \cite{Boisbunon} discussed the equivalence between AIC and SURE.
Such a correspondence holds in the current setting as well.
Specifically, consider estimation of $M$ from $Y \sim {\rm N}_{n,q}(M,I_n,\Sigma)$ under the loss $L(M,\hat{M})={\rm tr} ( \Sigma^{-1} (\hat{M}-M)^{\top} (\hat{M}-M))$.
Then, SURE for the estimator $\hat{M}=X \hat{B}$ is
\begin{align*}
	\lambda^{{\rm U}}(Y) &= {\rm tr} (\Sigma^{-1} (Y-X \hat{B})^{\top} (Y-X \hat{B})) + (2p-n)q \\
	&= {\rm AIC} - nq \log(2 \pi e) - n \log \det {\Sigma},
\end{align*}
and the loss is related to the Kullback--Leibler discrepancy as 
\begin{align*}
	L(M,\hat{M})=d( (B,\Sigma), (\hat{B},\Sigma) ) - nq \log(2 \pi e) - n \log \det {\Sigma}.
\end{align*}
Thus,  estimation of the loss $L(M,\hat{M})$ for the estimator $\hat{M}=X \hat{B}$ is equivalent to estimation of the Kullback--Leibler discrepancy for the plug-in predictive distribution, and both SURE and AIC are exactly unbiased.
Under this correspondence, Proposition~\ref{th_aic} is rewritten as follows.

\begin{corollary}
	For the multivariate linear regression model \eqref{fullmodel} with known $\Sigma \succ O$, consider the estimator $\hat{M} = X \hat{B}$ of $M=XB$ under the loss $L(M,\hat{M})={\rm tr} ( \Sigma^{-1} (\hat{M}-M)^{\top} (\hat{M}-M))$.
	If $p \geq 2q+3$, then SURE is inadmissible and dominated by the loss estimator
	\begin{align*}
		\lambda(Y) = \lambda^{{\rm U}}(Y)-\frac{2(p-2q-2)}{q} {\rm tr} (\Sigma ((X \hat{B})^{\top} (X \hat{B}))^{-1}).
	\end{align*}
\end{corollary}


\subsection{Multivariate linear regression with unknown covariance}
Next, consider the multivariate linear regression model \eqref{fullmodel} with unknown covariance $\Sigma \succ O$.
The maximum likelihood estimate is $\hat{B}=(X^{\top} X)^{-1} X^{\top} Y$ and $\hat{\Sigma}=(Y-X\hat{B})^{\top}(Y-X\hat{B})/n$.
From \eqref{Delta}, the Kullback--Leibler discrepancy is
\begin{align*}
	d ( (B,\Sigma), (\hat{B},\hat{\Sigma}) ) = n q \log (2 \pi) + n \log \det \hat{\Sigma} + n {\rm tr} (\hat{\Sigma}^{-1} \Sigma) + {\rm tr} (\hat{\Sigma}^{-1} (\hat{B}-B)^{\top} X^{\top} X (\hat{B}-B)).
\end{align*}
The AIC is
\begin{align*}
	{\rm AIC} = nq \log(2 \pi) + n \log \det \hat{\Sigma} + 2 \left( pq + \frac{q(q+1)}{2} \right).
\end{align*}
The corrected AIC is
\begin{align*}
	{\rm AICc} = nq \log(2 \pi) + n \log \det \hat{\Sigma} + \frac{2n}{n-p-q-1} \left( pq + \frac{q(q+1)}{2} \right).
\end{align*}
The corrected AIC is exactly unbiased while AIC is biased \citep{Hurvich,Sugiura,Bedrick}.
Then, we obtain the following.

\begin{theorem}\label{th_aic2}
	For the multivariate linear regression model \eqref{fullmodel} with unknown $\Sigma \succ O$, AIC is inadmissible and dominated by AICc as an estimator of the Kullback--Leibler discrepancy.
\end{theorem}
\begin{proof}
	For two random variables $S$ and $T$, we have
	\begin{align*}
		&{\rm E} [(S-T)^2] \\
		=& {\rm E} [(S-{\rm E} [S]+{\rm E} [S]-{\rm E} [T]+{\rm E} [T]-T)^2] \\
		=& {\rm E} [(S-{\rm E} [S])^2] + ({\rm E} [S]-{\rm E}[T])^2 + {\rm E} [(T-{\rm E} [T])^2] - 2 {\rm E} [(S-{\rm E} [S]) (T-{\rm E} [T])] \\
		=& {\rm Var} [S] + ({\rm E} [S]-{\rm E}[T])^2 + {\rm Var} [T] - 2 {\rm Cov} [S,T]. 
	\end{align*}
	Hence, the mean squared error of AIC is given by
	\begin{align}
		{\rm E} [({\rm AIC}-d)^2] = {\rm Var} [{\rm AIC}] + ({\rm E} [{\rm AIC}]-{\rm E}[d])^2 + {\rm Var} [d]- 2 {\rm Cov} [{\rm AIC},d], \label{aic_risk}
	\end{align}
	where we write $d ( (B,\Sigma), (\hat{B},\hat{\Sigma}) )$ by $d$ for simplicity.
	Similarly, the mean squared error of AICc is
	\begin{align}
		{\rm E} [({\rm AICc}-d)^2] = {\rm Var} [{\rm AICc}] + {\rm Var} [d]- 2 {\rm Cov} [{\rm AICc},d], \label{aicc_risk}
	\end{align}
	where we used ${\rm E} [{\rm AICc}]={\rm E}[d]$.
	
	On the other hand, since the difference between AIC and AICc is constant, 
	\begin{align}
		{\rm Var} [{\rm AIC}] = {\rm Var} [{\rm AICc}], \quad {\rm Cov} [{\rm AIC},d] = {\rm Cov} [{\rm AICc},d]. \label{aic_aicc}
	\end{align}
	
	From \eqref{aic_risk}, \eqref{aicc_risk} and \eqref{aic_aicc},
	\begin{align*}
		{\rm E}_{B,\Sigma} [({\rm AICc}-d)^2] \leq {\rm E}_{B,\Sigma} [({\rm AIC}-d)^2]
	\end{align*}
	for every $B$ and $\Sigma$.
\end{proof}

In the next section, we show that the corrected AIC is still inadmissible and provide improved loss estimators.
Note that \cite{Davies} showed that the corrected AIC is the minimum variance unbiased estimator of the expected Kullback--Leibler discrepancy. 

\section{Inadmissibility of the corrected AIC}\label{sec:main}


For the multivariate linear regression model \eqref{fullmodel} with unknown covariance, the corrected AIC is the minimum variance unbiased estimator of the expected Kullback--Leibler discrepancy from the Lehmann--Scheff\'e theorem \citep{Davies}.  
Also, Theorem~\ref{th_aic2} showed that the AIC is dominated by the corrected AIC as an estimator of the Kullback--Leibler discrepancy.
However, the corrected AIC is still inadmissible as follows, where MAICc is an abbreviation of ``Modified AICc."

\begin{theorem}\label{th_main}
	Consider the multivariate linear regression model \eqref{fullmodel} with unknown $\Sigma \succ O$.
	Let
	\begin{align*}
		\bar{c} = \frac{4 n^2}{(n-p)(q(n-p)+2)} \left( p-2q-2- \frac{q^2+q-2}{n-p-q-1} \right). 
	\end{align*}
	If $n-p-q-1>0$ and $\bar{c}>0$, then for any $c \in (0,\bar{c}]$, AICc is inadmissible and dominated by
	\begin{align}
		{\rm MAICc} = {\rm AICc} - c {\rm tr} (\hat{\Sigma} ((X \hat{B})^{\top} (X \hat{B}))^{-1}) \label{MAICc2}
	\end{align}
	as an estimator of the Kullback--Leibler discrepancy. 
\end{theorem}
\begin{proof}
	From the standard theory of multivariate linear regression \citep{Anderson}, the maximum likelihood estimates $\hat{B}=(X^{\top} X)^{-1} X^{\top} Y$ 
	and $\hat{\Sigma}= (Y-X \hat{B})^{\top} (Y-X \hat{B})/n$ for \eqref{fullmodel} are independently distributed as 
	\begin{align*}
		\hat{B} \sim {\rm N}_{p,q} (B, (X^{\top} X)^{-1}, \Sigma), \quad \hat{\Sigma} \sim W_q \left(n-p,\frac{1}{n} \Sigma \right). 
	\end{align*}
	Thus, $Z = (X^{\top} X)^{1/2} \hat{B} \Sigma^{-1/2}$ and $S=\Sigma^{-1/2} \hat{\Sigma} \Sigma^{-1/2}$ are independetly distributed as 
	\begin{align*}
		Z \sim {\rm N}_{p,q}(\bar{Z},I_p,I_q), \quad S \sim W_q \left( n-p, \frac{1}{n} I_q \right),
	\end{align*}
	where $\bar{Z}=(X^{\top} X)^{1/2} {B} \Sigma^{-1/2}$.
	
	Again, we write $d ( (B,\Sigma), (\hat{B},\hat{\Sigma}) )$ in \eqref{Delta} as $d$ for simplicity.
	Let $d = d( (B,\Sigma), (\hat{B},\hat{\Sigma}) )$ and
	\begin{align*}
		h={\rm MAICc}-{\rm AICc} = - c {\rm tr} (\hat{\Sigma} ((X \hat{B})^{\top} (X \hat{B}))^{-1}).
	\end{align*}
	Then,
	\begin{align}
		{\rm E}_{B,\Sigma} [ ( {\rm MAICc}-d )^2] - {\rm E}_{B,\Sigma} [ ({\rm AICc}-d)^2 ] = {\rm E}_{B,\Sigma} [ h^2 + 2h ({\rm AICc}-d) ]. \label{risk_diff}
	\end{align}
	We evaluate each term.
	Note that $h=-c {\rm tr} (S (Z^{\top} Z)^{-1})$, since
	\begin{align*}
		{\rm tr} (\hat{\Sigma} ((X \hat{B})^{\top} (X \hat{B}))^{-1}) &= {\rm tr} (S {\Sigma}^{1/2} ((X \hat{B})^{\top} (X \hat{B}))^{-1} {\Sigma}^{1/2}) \\
		&= {\rm tr} (S ({\Sigma}^{-1/2} \hat{B}^{\top} X^{\top} X \hat{B} {\Sigma}^{-1/2} )^{-1}) \\
		&= {\rm tr} (S (Z^{\top} Z)^{-1}).
	\end{align*}
	In the following, we write ${\rm E}_{B,\Sigma}$ as ${\rm E}$ for simplicity.
	
	First, by using Lemma~\ref{lem_gupta3},
	\begin{align}
		{\rm E} [h^2] &=  c^2 {\rm E} [{\rm tr} (S (Z^{\top} Z)^{-1}) {\rm tr} (S (Z^{\top} Z)^{-1})] \nonumber \\
		&=  \frac{2(n-p)}{n^2} c^2 {\rm E} [{\rm tr} ((Z^{\top} Z)^{-2})] + \frac{(n-p)^2}{n^2} c^2 {\rm E} [ ({\rm tr} ((Z^{\top} Z)^{-1}))^2]. \label{h2}
	\end{align}
	
	Next, from
	\begin{align*}
		{\rm AICc} &= nq \log (2\pi) + n \log \det \hat{\Sigma} + \frac{nq(n+p)}{n-p-q-1},
	\end{align*}
	\begin{align*}
		d &= nq \log (2\pi) + n \log \det \hat{\Sigma} + n {\rm tr} (\hat{\Sigma}^{-1} \Sigma) +  {\rm tr} (X (\hat{B}-B) \hat{\Sigma}^{-1} (\hat{B}-B)^{\top} X^{\top}) \\
		&= nq \log (2\pi) + n \log \det \hat{\Sigma} + n {\rm tr} (S^{-1}) +  {\rm tr} ((Z-\bar{Z})^{\top} (Z-\bar{Z}) S^{-1}),
	\end{align*}
	we have
	\begin{align}
		{\rm E} [2h ({\rm AICc}-d)] =&  -\frac{2n(n+p)q}{n-p-q-1} c {\rm E} [{\rm tr} (S (Z^{\top} Z)^{-1})] + 2 n c {\rm E} [{\rm tr} (S^{-1}) {\rm tr} (S (Z^{\top} Z)^{-1})] \nonumber \\
		& \quad + 2 c {\rm E} [{\rm tr} ((Z-\bar{Z})^{\top} (Z-\bar{Z}) S^{-1}) {\rm tr} (S (Z^{\top} Z)^{-1})]. \label{term}
	\end{align}
	Using the independence of $S$ and $Z$,
	\begin{align}
		{\rm E} [{\rm tr} (S (Z^{\top} Z)^{-1})] &= {\rm tr} ({\rm E} [S] \cdot {\rm E} [(Z^{\top} Z)^{-1})]) \nonumber \\
		&= {\rm tr} \left( \frac{n-p}{n} I_q \cdot {\rm E} [(Z^{\top} Z)^{-1})] \right) \nonumber \\
		&= \frac{n-p}{n} {\rm E} [{\rm tr} ( (Z^{\top} Z)^{-1} )]. \label{term1}
	\end{align}
	Similarly, using the independence of $S$ and $Z$ and Lemma~\ref{lem_gupta1},
	\begin{align}
		{\rm E} [{\rm tr} (S^{-1}) {\rm tr} (S (Z^{\top} Z)^{-1})] &=  {\rm tr} ({\rm E} [{\rm tr} (S^{-1}) S] \cdot {\rm E} [(Z^{\top} Z)^{-1}]) \nonumber \\
		&=  {\rm tr} \left( \frac{(n-p)q-2}{n-p-q-1} I_q \cdot {\rm E} [(Z^{\top} Z)^{-1}] \right) \nonumber \\	
		&= \frac{(n-p)q-2}{n-p-q-1} {\rm E} [{\rm tr} ( (Z^{\top} Z)^{-1} )]. \label{term2}
	\end{align}
	Also, from Lemma~\ref{lem_exp},
	\begin{align}
		&{\rm E} [{\rm tr} ((Z-\bar{Z})^{\top} (Z-\bar{Z}) S^{-1}) {\rm tr} (S (Z^{\top} Z)^{-1})] \nonumber \\
		=& p \frac{(n-p)q-2}{n-p-q-1} {\rm E} [ {\rm tr} (  (Z^{\top} Z)^{-1}  ) ] - 2 \frac{(n-p-1)(p-q-2)+2}{n-p-q-1} {\rm E} [ {\rm tr} ( (Z^{\top} Z)^{-2} ) ] \nonumber \\
		& \quad +2 \frac{n-q-2}{n-p-q-1} {\rm E} [ ({\rm tr} ((Z^{\top} Z)^{-1}) )^2 ]. \label{term3}
	\end{align}
	Therefore, by substituting \eqref{term1}, \eqref{term2} and \eqref{term3} into \eqref{term}, 
	\begin{align}
		&{\rm E} [2h ({\rm AICc}-d)] \nonumber \\
		=& \frac{-2(n-p)(n+p)q+2n((n-p)q-2)+2p((n-p)q-2)}{n-p-q-1} c {\rm E} [{\rm tr} ( (Z^{\top} Z)^{-1} )] \nonumber \\
		& \ - 4c \frac{(n-p-1)(p-q-2)+2}{n-p-q-1} {\rm E} [ {\rm tr} ( (Z^{\top} Z)^{-2} ) ] +4c \frac{n-q-2}{n-p-q-1} {\rm E} [ ({\rm tr} ((Z^{\top} Z)^{-1}) )^2 ] \nonumber \\
		=& -4c \frac{n+p}{n-p-q-1} {\rm E} [{\rm tr} ( (Z^{\top} Z)^{-1} )] \nonumber \\
		& \ - 4c \frac{(n-p-1)(p-q-2)+2}{n-p-q-1} {\rm E} [ {\rm tr} ( (Z^{\top} Z)^{-2} ) ] +4c \frac{n-q-2}{n-p-q-1} {\rm E} [ ({\rm tr} ((Z^{\top} Z)^{-1}) )^2 ].  \label{hAIC}
	\end{align}
	
	Hence, by substituting \eqref{h2} and \eqref{hAIC} into \eqref{risk_diff},
	\begin{align*}
		&{\rm E}_{B,\Sigma} [ ( {\rm MAICc}-d )^2] - {\rm E}_{B,\Sigma} [ ({\rm AICc}-d)^2 ] \\
		=& -4c \frac{n+p}{n-p-q-1} {\rm E} [{\rm tr} ( (Z^{\top} Z)^{-1} )] \\
		& \quad + \left( \frac{2(n-p)}{n^2} c^2 - 4c \frac{(n-p-1)(p-q-2)+2}{n-p-q-1} \right) {\rm E} [ {\rm tr} ( (Z^{\top} Z)^{-2} ) ] \\
		& \quad + \left( \frac{(n-p)^2}{n^2} c^2 + 4c \frac{n-q-2}{n-p-q-1} \right) {\rm E} [ ({\rm tr} ((Z^{\top} Z)^{-1}) )^2 ] \\
		\leq& -4c \frac{n+p}{n-p-q-1} {\rm E} [{\rm tr} ( (Z^{\top} Z)^{-1} )] \\
		& \quad + c \left( \frac{(n-p)(q(n-p)+2)}{n^2} c - 4 \frac{(n-p-1)(p-q-2)-q(n-q-2)+2}{n-p-q-1} \right) {\rm E} [ {\rm tr} ( (Z^{\top} Z)^{-2} ) ] \\
		=& -4c \frac{n+p}{n-p-q-1} {\rm E} [{\rm tr} ( (Z^{\top} Z)^{-1} )] + \frac{(n-p)(q(n-p)+2)}{n^2} c (c-\bar{c}) {\rm E} [ {\rm tr} ( (Z^{\top} Z)^{-2} ) ],
	\end{align*}
	where we used $({\rm tr} ((Z^{\top} Z)^{-1}) )^2 \leq q {\rm tr} ((Z^{\top} Z)^{-2})$ from \eqref{cs} and $(n-p-1)(p-q-2)-q(n-q-2)=(n-p-q-1)(p-2q-2)-q^2-q$.
	Therefore, if $0 < c \leq \bar{c}$, then
	\begin{align*}
		{\rm E}_{B,\Sigma} [ ( {\rm MAICc}-d )^2] \leq {\rm E}_{B,\Sigma} [ ({\rm AICc}-d)^2 ]
	\end{align*}
	for every $B$ and $\Sigma$.
	
	\end{proof}
	
	Note that, when $n$ is sufficiently large, the condition $\bar{c}>0$ in Theorem~\ref{th_main} is reduced to $p \geq 2q+3$, which is the same as in Theorem~\ref{th_ms2} and Theorem~\ref{th_aic}.

					In the case of a single response ($q=1$), Theorem~\ref{th_main} is reexpressed as follows.
					Here, we employ the usual notation for the linear regression with a single response:
					\begin{align}
						y \sim {\rm N}_n (X \beta, \sigma^2 I_n), \quad \hat{\beta} = (X^{\top} X)^{-1} X^{\top} y, \quad \hat{\sigma}^2 = \| y-X\hat{\beta} \|^2/n. \label{single}
					\end{align}
					
					\begin{corollary}\label{cor_main}
						Consider the linear regression model $y \sim {\rm N}_n (X \beta, \sigma^2 I_n)$ with unknown $\sigma^2$.
						Let
						\begin{align*}
							\bar{c} = \frac{4 n^2 (p-4)}{(n-p) (n-p+2)}.
						\end{align*}
						If $n-p-2>0$ and $\bar{c} > 0$, then for any $c \in (0,\bar{c}]$, ${\rm AICc}$ is inadmissible and dominated by
						\begin{align*}
							{\rm MAICc} = {\rm AICc} - c \hat{\sigma}^2 \| X \hat{\beta} \|^{-2}
						\end{align*}
						as an estimator of the Kullback--Leibler discrepancy. 
					\end{corollary}
					
					Note that the condition $\bar{c}>0$ in Corollary~\ref{cor_main} is equivalent to $p \geq 5$, as in Proposition~\ref{prop_johnstone}.

					\section{Numerical results}\label{sec:experiments}
					
					\subsection{Single response}
					First, we consider the case of single response \eqref{single}, which corresponds to the model \eqref{fullmodel} with $q=1$ and Corollary~\ref{cor_main}.
					We compare the mean squared errors of AICc and MAICc by Monte Carlo experiments with $10^6$ repetitions. 
					Each entry of $X$ is generated from ${\rm N}(0,1)$ independently.
					We plot the percentage improvement in mean squared error (MSE) of MAICc over AICc: 
					\begin{align*}
						100 \frac{{\rm E}_{\beta,\sigma^2} [({\rm AICc}-d)^2]-{\rm E}_{\beta,\sigma^2} [({\rm MAICc}-d)^2]}{{\rm E}_{\beta,\sigma^2} [({\rm AICc}-d)^2]},
					\end{align*}
					where $d=d ( (\beta,\sigma^2), (\hat{\beta},\hat{\sigma}^2) )$ is the Kullback--Leibler discrepancy \eqref{Delta}.
					
					Figure~\ref{fig_c_uni} compares MAICc with different values of $c$ for $n=30$, $p=10$ and $\sigma^2=1$.
					The left panel indicates that MAICc with $c=\bar{c}$ dominates MAICc with $c=0.5 \bar{c}$, whereas the right panel shows that the mean squared error of MAICc at $\beta=0$ attains its minimum around $c = 0.8 \bar{c}$.
					Overall, setting $c=\bar{c}$ in MAICc seems to be a reasonable choice.
					Thus, we adopt this value of $c$ in the following experiments.
					
					\begin{figure}
						\centering
						\includegraphics{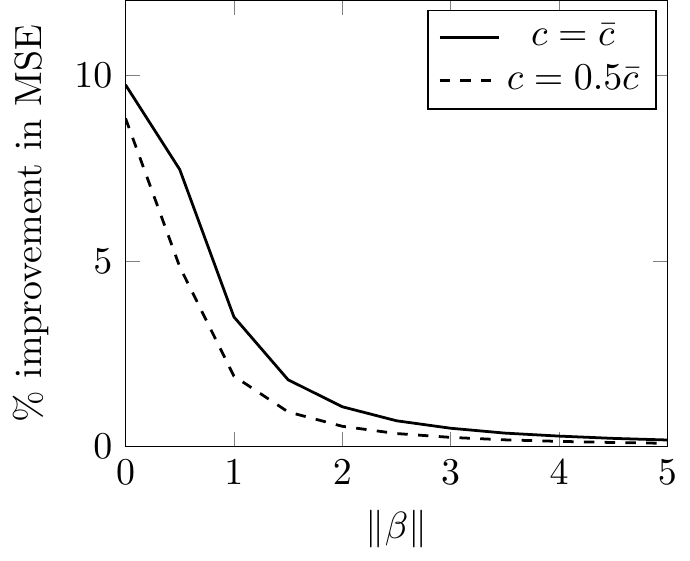}
						\includegraphics{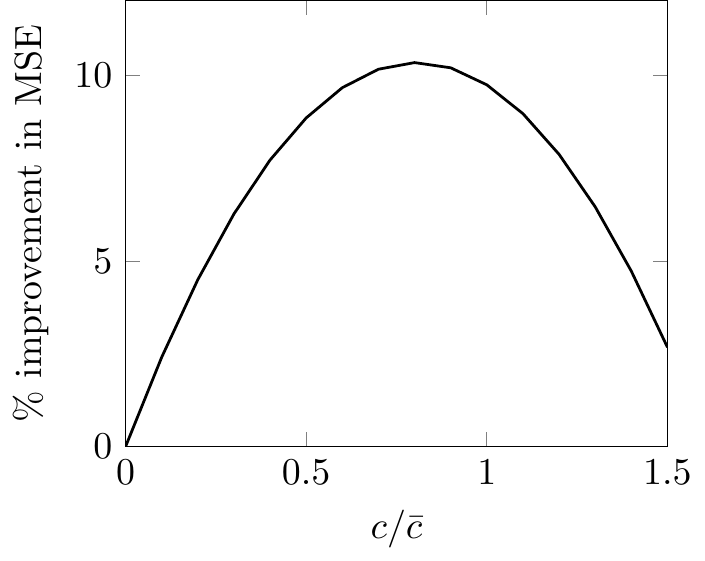}
						\caption{Percentage improvement in mean squared error of MAICc over AICc when $n=30$, $p=10$ and $\sigma^2=1$. Left: Plot with respect to $\| \beta \|$ for $c=\bar{c}$ and $c=0.5 \bar{c}$. Right: Plot with respect to $c/\bar{c}$ for $\beta=0$.}
						\label{fig_c_uni}
					\end{figure}
					
					Figure~\ref{fig_n_uni} compares the performance of MAICc for different values of $n$ when $p=10$ and $\sigma^2=1$.
					It indicates that the percentage improvement in MSE is larger for smaller $n$.
					Figure~\ref{fig_p_uni} compares the performance of MAICc for different values of $p$ when $n=30$ and $\sigma^2=1$.
					It indicates that the percentage improvement in MSE is maximized around $p=15$.
					Figure~\ref{fig_sigma2_uni} compares the performance of MAICc for different values of $\sigma^2$ when $n=30$ and $p=10$.
					It indicates that the percentage improvement in MSE is larger for larger $\sigma^2$ at $\beta \neq 0$. Note that the percentage improvement in MSE at $\beta=0$ does not depend on $\sigma^2$.
					
					\begin{figure}
						\centering
						\includegraphics{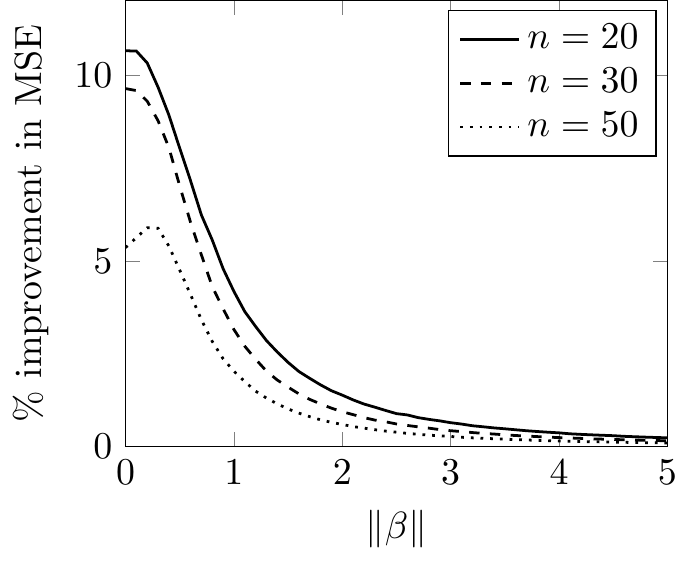}
						\includegraphics{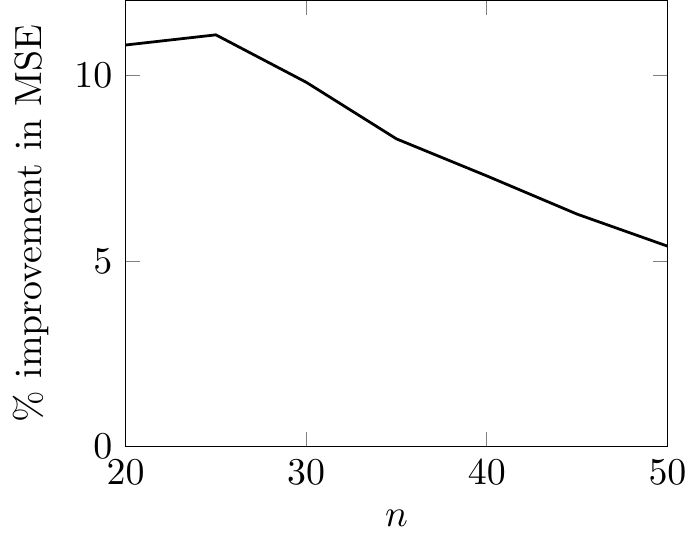}
						\caption{Percentage improvement in mean squared error of MAICc with $c=\bar{c}$ over AICc when $p=10$ and $\sigma^2=1$. Left: Plot with respect to $\| \beta \|$ for $n=30, 50, 100$. Right: Plot with respect to $n$ for $\beta=0$.}
						\label{fig_n_uni}
					\end{figure}
					
					\begin{figure}
						\centering
						\includegraphics{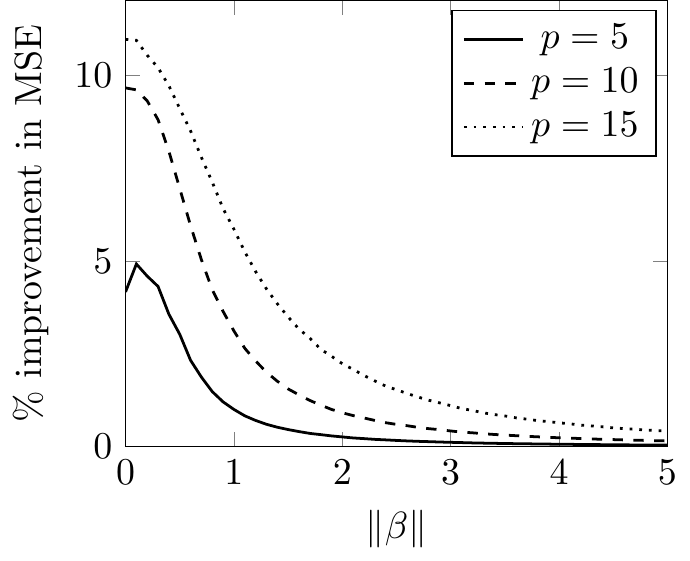}
						\includegraphics{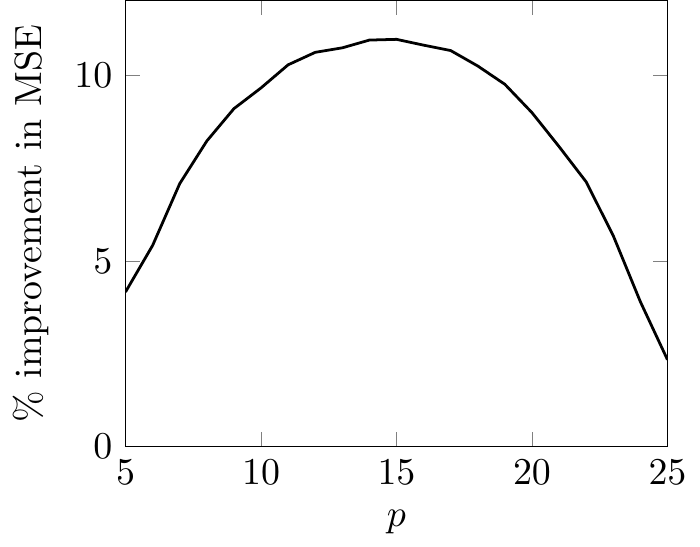}
						\caption{Percentage improvement in mean squared error of MAICc with $c=\bar{c}$ over AICc when $n=30$ and $\sigma^2=1$. Left: Plot with respect to $\| \beta \|$ for $p=5, 10, 20$. Right: Plot with respect to $p$ for $\beta=0$.}
						\label{fig_p_uni}
					\end{figure}
					
					\begin{figure}
						\centering
						\includegraphics{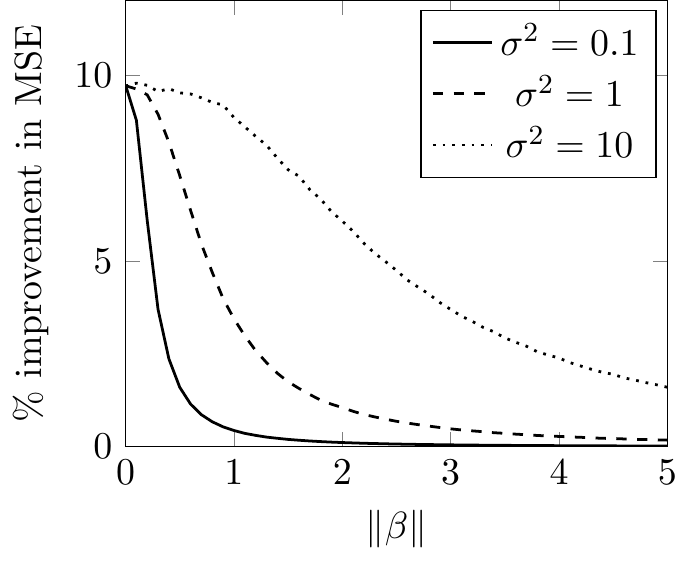}
						\caption{Percentage improvement in mean squared error of MAICc with $c=\bar{c}$ over AICc when $n=30$ and $p=10$.}
						\label{fig_sigma2_uni}
					\end{figure}
					
					\subsection{Multi-response}
					Now, consider the multi-response case \eqref{fullmodel}.
					As in the previous subsection, we compare the mean squared error of AICc and MAICc by Monte Carlo experiments with $10^6$ repetitions. 
					Each entry of $X$ is generated from ${\rm N}(0,1)$ independently.
					We plot the percentage improvement in mean squared error (MSE) of MAICc over AICc: 
					\begin{align*}
						100 \frac{{\rm E}_{B,\Sigma} [({\rm AICc}-d)^2]-{\rm E}_{B,\Sigma} [({\rm MAICc}-d)^2]}{{\rm E}_{B,\Sigma} [({\rm AICc}-d)^2]},
					\end{align*}
					where $d=d ( (B,\Sigma), (\hat{B},\hat{\Sigma}) )$ is the Kullback--Leibler discrepancy \eqref{Delta}.
					
					Figure~\ref{fig_c} compares MAICc with different values of $c$ for $n=30$, $p=10$ and $q=2$.
					The improvement is large when some of the singular values of $M$ are small. 
					Similarly to Figure~\ref{fig_matsuda}, MAICc attains constant reduction of MSE as long as $\sigma_2(M)=0$, even when $\sigma_1(M)$ is large.
					Thus, MAICc works well when $B$ is close to low-rank, which corresponds to reduced-rank regression \citep{Reinsel}.
					We found that MAICc with $c=\bar{c}$ is numerically dominated by MAICc with $c=2 \bar{c}$, which implies that the upper bound for $c$ in Theorem~\ref{th_main} may be improved.
					We adopt $c=\bar{c}$ in the following experiments.
					
					\begin{figure}
						\centering
						\includegraphics{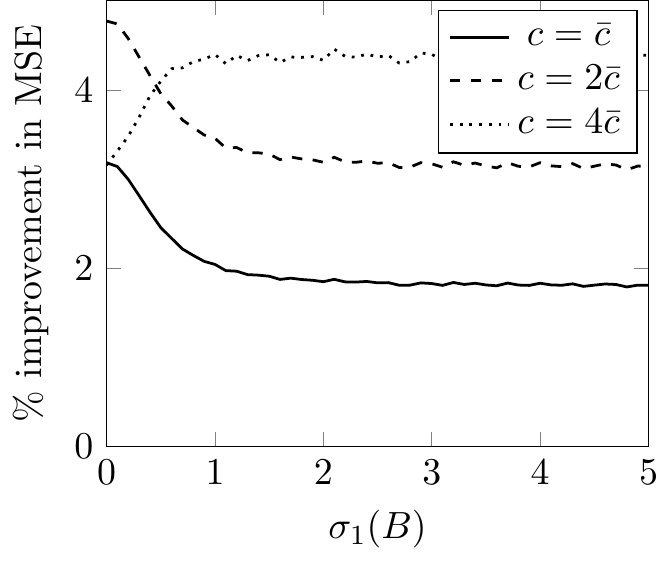}
						\includegraphics{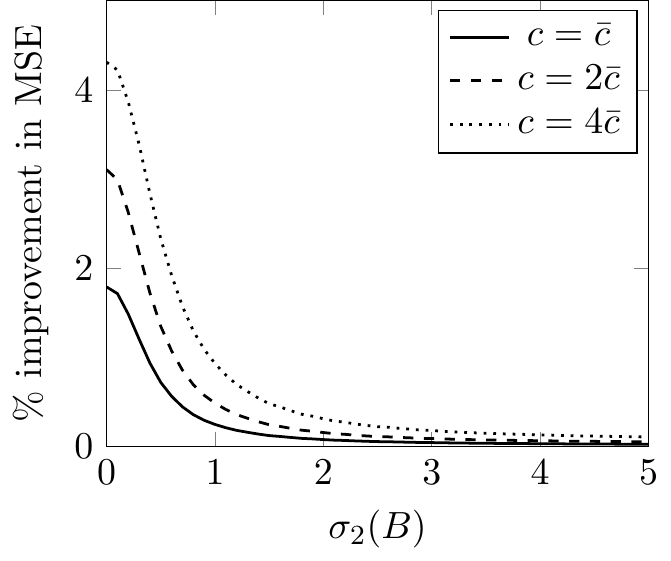}
						\includegraphics{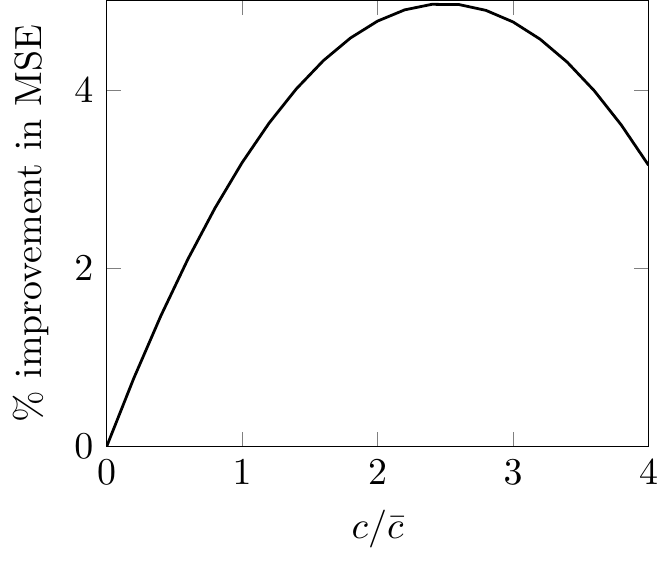}
						\includegraphics{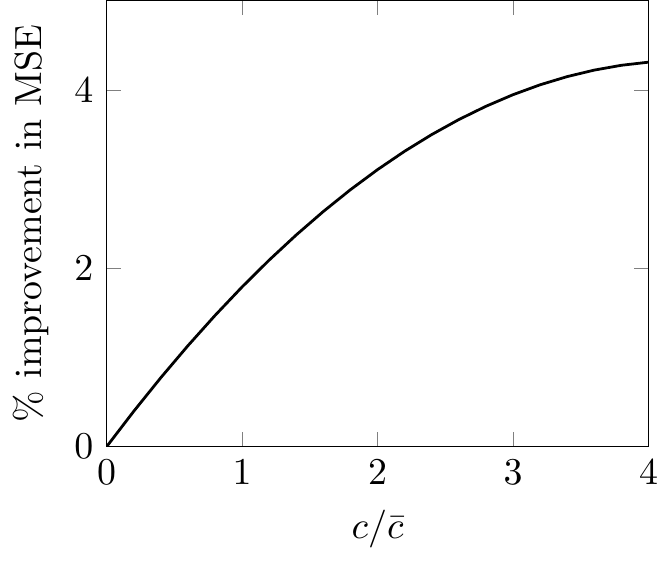}
						\caption{Percentage improvement in mean squared error of MAICc over AICc when $n=30$, $p=10$, $q=2$ and $\Sigma=I_2$. Upper left: Plot with respect to $\sigma_1(B)$ for $c=\bar{c}$, $c=2 \bar{c}$ and $c=4 \bar{c}$ when $\sigma_2(B)=0$. Upper right: Plot with respect to $\sigma_2(B)$ for $c=\bar{c}$, $c=2 \bar{c}$ and $c=4 \bar{c}$ when $\sigma_1(B)=5$. Lower left: Plot with respect to $c/\bar{c}$ for $\sigma_1(B)=\sigma_2(B)=0$ ($B=O$). Lower right: Plot with respect to $c/\bar{c}$ for $\sigma_1(B)=5$ and $\sigma_2(B)=0$.}
						\label{fig_c}
					\end{figure}
					
					Figure~\ref{fig_n} compares the performance of MAICc for different values of $n$ when $p=10$, $q=2$ and $\Sigma=I_2$.
					It indicates that the percentage improvement in MSE is maximized around $n=40$.
					Figure~\ref{fig_p} compares the performance of MAICc for different values of $p$ when $n=30$, $q=2$ and $\Sigma=I_2$.
					It indicates that the percentage improvement in MSE is smaller for larger $p$.
					Figure~\ref{fig_rq} (left) compares the performance of MAICc for different values of $r=\Sigma_{12}$ when $n=30$, $p=10$, $q=2$ and $\Sigma_{11}=\Sigma_{22}=1$.
					It indicates that the percentage improvement in MSE is largest for $r=0$ (no correlation).
					
					\begin{figure}
						\centering
						\includegraphics{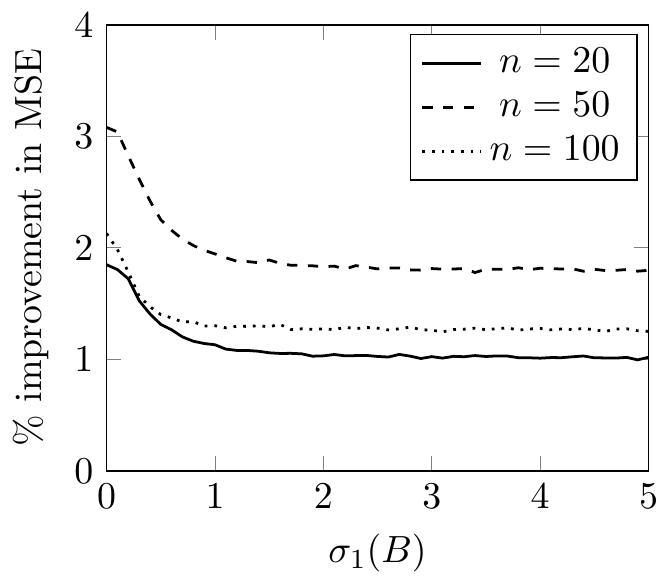}
						\includegraphics{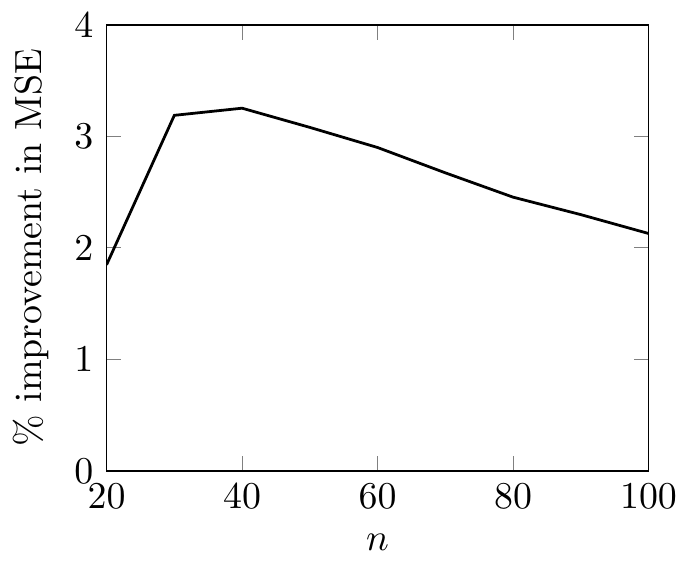}
						\caption{Percentage improvement in mean squared error of MAICc with $c=\bar{c}$ over AICc when $p=10$, $q=2$ and $\Sigma=I_2$. Left: Plot with respect to $\sigma_1(B)$ for $n=20,50,100$ when $\sigma_2(B)=0$. Right: Plot with respect to $n$ for $B=O$.}
						\label{fig_n}
					\end{figure}
					
					\begin{figure}
						\centering
						\includegraphics{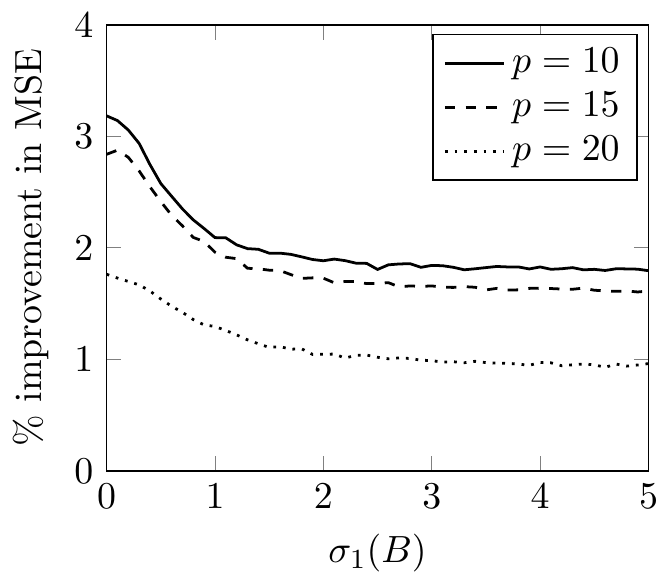}
						\includegraphics{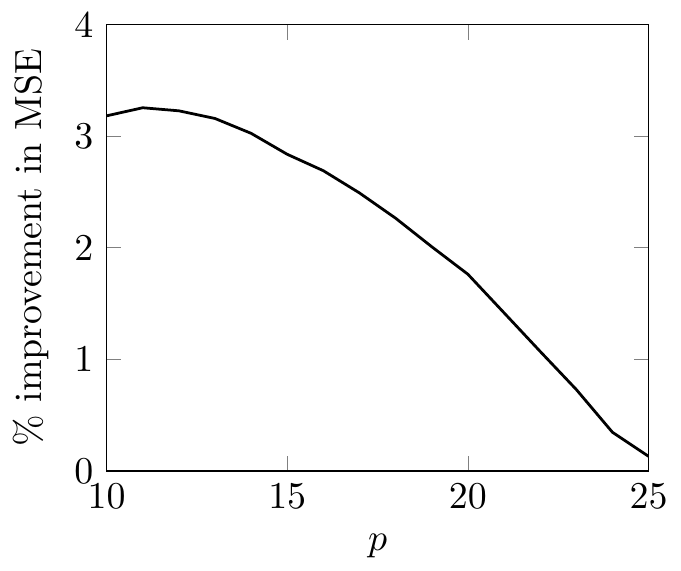}
						\caption{Percentage improvement in mean squared error of MAICc with $c=\bar{c}$ over AICc when $n=30$, $q=2$ and $\Sigma=I_2$. Left: Plot with respect to $\sigma_1(B)$ for $p=10,15,20$ when $\sigma_2(B)=0$. Right: Plot with respect to $p$ for $B=O$.}
						\label{fig_p}
					\end{figure}
					
					\begin{figure}
						\centering
						\includegraphics{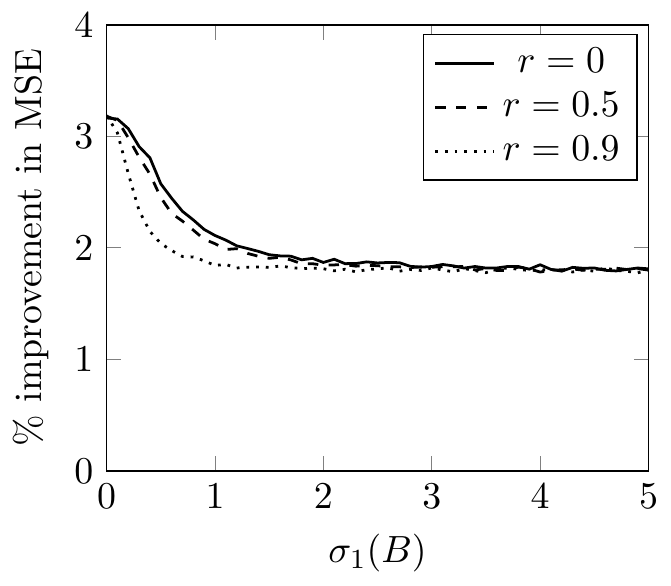}
						\includegraphics{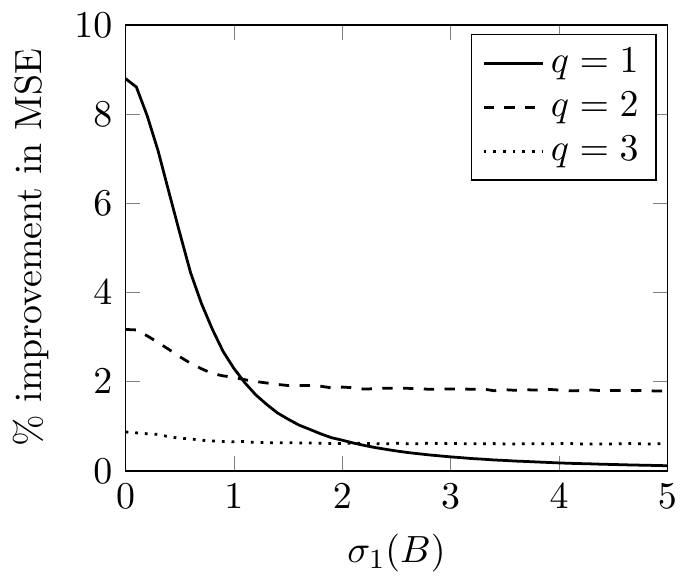}
						\caption{Left: Percentage improvement in mean squared error of MAICc with $c=\bar{c}$ over AICc when $n=30$, $p=10$, $q=2$, $\Sigma_{11}=\Sigma_{22}=1$ and $\Sigma_{12}=r$. Plot with respect to $r$. Right: Percentage improvement in mean squared error of MAICc with $c=\bar{c}$ over AICc when $n=30$, $p=10$ and $\Sigma=I_q$. Plot with respect to $\sigma_1(B)$ for $q=1,2,3$ when $\sigma_2(B)=\dots=\sigma_q(B)=0$.}
						\label{fig_rq}
					\end{figure}
					
					Finally, Figure~\ref{fig_rq} (right) compares the performance of MAICc for different values of $q$ when $n=30$, $p=10$, $\Sigma=I_q$.
					Compared to the single response case ($q=1$) in the previous subsection, the percentage improvement in MSE of MAICc over AICc is not large.
					We expect that MAICc for $q \geq 2$ can be improved in several ways.
					For example, as shown in Figure~\ref{fig_c}, MAICc with $c=\bar{c}$ is numerically dominated by MAICc with larger value of $c$.
					Thus, improving the upper bound of $c$ in Theorem~\ref{th_main} would be beneficial.
					Also, in analogy to the method of \cite{Efron76} for improving the Efron--Morris estimator by adding scalar shrinkage, MAICc is expected to be improved by adding a term of the form in Propostion~\ref{prop_johnstone} after vectorization.
					Another solution may be to use different coefficients for the singular values as in Proposition~\ref{prop_ms}.
					
					
					\subsection{Variable selection}
					
					Here, we compare the variable selection performance of ${\rm AIC}$, ${\rm AICc}$, and ${\rm MAICc}$ with $c=\bar{c}$. 
					The experimental setting is similar to that of \cite{Hurvich}: $n=20$, $p=10$, $q=1$, $\beta=(0.1,0.2,0.3,0.4,0.5,0,0,0,0,0)^{\top}$ and $\sigma^2=1$.
					Each entry of $X$ is generated from ${\rm N}(0,1)$ independently and fixed for the whole experiment.
					Ten submodels were considered as candidate models, where the $k$-th model uses the first $k$ columns of $X$ as covariates ($k=1,\dots,10$).
					Thus, the fifth model is the truth here.
					For each realization of $y$, we selected the model order $k$ by minimizing ${\rm AIC}$, ${\rm AICc}$, or ${\rm MAICc}$ with $c=c^{*}$.
					Table \ref{tab_reg10} shows the frequency of the selected order in 1000 realizations.
					${\rm MAICc}$ selects the true order $k=5$ more frequently than ${\rm AIC}$ and ${\rm AICc}$.
					Thus, ${\rm MAICc}$ attains better performance in variable selection than ${\rm AIC}$ and ${\rm AICc}$.
					This result indicates a practical advantage of introducing the current loss estimation framework to investigate information criterion.
					
					\begin{table}[htbp]
						\centering
						\caption{Frequency of order selected by three criteria in 1000 realizations} 
					\begin{tabular}{|c|c|c|c|c|c|c|c|c|c|c|}
						\hline
						& 1 & 2 & 3 & 4 & 5 & 6 & 7 & 8 & 9 & 10 \\ \hline
						${\rm AIC}$ &  89 & 8 & 15 & 29 & 352 & 129 & 76 & 76 & 81 & 145 \\ \hline 
						${\rm AICc}$ & 277 & 147 & 37 & 16 & 460 & 44 & 15 & 4 & 0 & 0 \\ \hline 
						${\rm MAICc}$ & 248 & 137 & 34 & 14 & 492 & 54 & 17 & 4 & 0 & 0 \\ \hline 
					\end{tabular}
					\label{tab_reg10}
				\end{table}
				
				\section{Conclusion}\label{sec:concl}
				In this study, we showed that the corrected AIC is inadmissible as an estimator of the Kullback--Leibler discrepancy and provided improved loss estimators.
				To the best of our knowledge, such a loss estimation framework has not been employed in the study of information criteria, and there are several possible directions for future research.
				For example, generalizations of the current results to out-of-sample prediction \citep{Rosset}, high-dimensional settings \citep{Bellec,Fujikoshi14,Yanagihara}, and mis-specified cases \citep{Fujikoshi,Reschenhofer} may be interesting.
				Also, whereas we focused on the Gaussian linear regression model in this study, similar results may be obtained in general models for AIC and other information criteria such as TIC and GIC \citep{Konishi08} by asymptotic arguments. 
				Improvement of model averaging criteria such as Mallows criterion \citep{Hansen,Wan10} is another future problem.
				Finally, whereas we focused on the plug-in predictive distribution in this study, it would be interesting to study extension to the Bayesian predictive distribution, which minimizes the Bayes risk under the Kullback--Leibler loss \citep{Aitchison}.
				For the linear regression model, \cite{Kitagawa97} derived an information criterion for the Bayesian predictive distribution and \cite{Kobayashi} studied the problem of Bayesian out-of-sample prediction.
				
	\section*{Acknowledgements}
					The author thanks the associate editor and referees for valuable comments.
					This work was supported by JSPS KAKENHI Grant Numbers 19K20220, 21H05205, 22K17865 and JST Moonshot Grant Number JPMJMS2024.

				\appendix
				
				\section{Matrix derivative formulas}
				
				\begin{lemma}\label{lem_diff}
					For $Z \in \mathbb{R}^{p \times q}$,
					\begin{align*}
						\frac{\partial}{\partial Z_{ij}} ((Z^{\top} Z)^{-1})_{kl} = - ((Z^{\top} Z)^{-1})_{kj} (Z(Z^{\top} Z)^{-1})_{il} - ((Z^{\top} Z)^{-1} Z^{\top})_{ki} ((Z^{\top} Z)^{-1})_{jl}.
					\end{align*}
				\end{lemma}
				\begin{proof}
					Let $\delta_{ab}$ be the Kronecker delta: $\delta_{ab}=1$ if $a=b$ and $\delta_{ab}=0$ if $a \neq b$.
					From
					\begin{align*}
						\frac{\partial}{\partial Z_{ij}} (Z^{\top} Z)_{ab} = \sum_c \frac{\partial}{\partial Z_{ij}} (Z_{ca} Z_{cb}) = \delta_{ja} Z_{ib} + \delta_{jb} Z_{ia}
					\end{align*}
					and ${\rm d} (A^{-1}) = -A^{-1} ({\rm d} A) A^{-1}$, we obtain
					\begin{align*}
						\frac{\partial}{\partial Z_{ij}} ((Z^{\top} Z)^{-1})_{kl} &= \left( -(Z^{\top} Z)^{-1} \left( \frac{\partial}{\partial Z_{ij}} (Z^{\top} Z) \right) (Z^{\top} Z)^{-1} \right)_{kl} \\
						&= - \sum_{a,b} ((Z^{\top} Z)^{-1})_{ka} \left( \frac{\partial}{\partial Z_{ij}} (Z^{\top} Z)_{ab} \right) ((Z^{\top} Z)^{-1})_{bl} \\
						&= - \sum_{a,b} ((Z^{\top} Z)^{-1})_{ka} ( \delta_{ja} Z_{ib} + \delta_{jb} Z_{ia}) ((Z^{\top} Z)^{-1})_{bl} \\
						&= - \sum_b (Z^{\top} Z)^{-1})_{kj} Z_{ib} ((Z^{\top} Z)^{-1})_{bl} - \sum_a ((Z^{\top} Z)^{-1})_{ka} Z_{ia} ((Z^{\top} Z)^{-1})_{jl} \\
						&= - ((Z^{\top} Z)^{-1})_{kj} (Z(Z^{\top} Z)^{-1})_{il} - ((Z^{\top} Z)^{-1} Z^{\top})_{ki} ((Z^{\top} Z)^{-1})_{jl}.
					\end{align*}
				\end{proof}
				
				\begin{lemma}\label{lem_diff2}
					For $Z \in \mathbb{R}^{p \times q}$ and $S \in \mathbb{R}^{q \times q}$,
					\begin{align*}
						\frac{\partial}{\partial Z_{ij}} {\rm tr} (S (Z^{\top} Z)^{-1}) = -2 (Z(Z^{\top} Z)^{-1} S (Z^{\top} Z)^{-1})_{ij}.
					\end{align*}
				\end{lemma}
				\begin{proof}
					From Lemma~\ref{lem_diff},
					\begin{align*}
						\frac{\partial}{\partial Z_{ij}} {\rm tr} (S (Z^{\top} Z)^{-1}) &= \sum_{k,l} S_{kl}  \frac{\partial}{\partial Z_{ij}} ((Z^{\top} Z)^{-1})_{kl}  \\
						&= \sum_{k,l} S_{kl} ( -((Z^{\top} Z)^{-1})_{kj} (Z(Z^{\top} Z)^{-1})_{il} - ((Z^{\top} Z)^{-1} Z^{\top})_{ki} ((Z^{\top} Z)^{-1})_{jl}) \\
						&= -2 (Z(Z^{\top} Z)^{-1} S (Z^{\top} Z)^{-1})_{ij}.
					\end{align*}
				\end{proof}
				
				\begin{lemma}\label{lem_diff3}
					For $Z \in \mathbb{R}^{p \times q}$ and $A,B \in \mathbb{R}^{q \times q}$,
					\begin{align*}
						&\sum_{ij} \frac{\partial}{\partial Z_{ij}} (Z (Z^{\top} Z)^{-1} A (Z^{\top} Z)^{-1} B)_{ij} \\
						=& (p-q-2) {\rm tr} ((Z^{\top} Z)^{-1} A (Z^{\top} Z)^{-1} B) - {\rm tr}  (A (Z^{\top} Z)^{-1}) {\rm tr}  ((Z^{\top} Z)^{-1} B).
					\end{align*}
				\end{lemma}
				\begin{proof}
					From Lemma~\ref{lem_diff},
					\begin{align*}
						&\frac{\partial}{\partial Z_{ij}} (Z (Z^{\top} Z)^{-1} A (Z^{\top} Z)^{-1} B)_{ij} \\
						=& \sum_k \left( \frac{\partial}{\partial Z_{ij}} Z_{ik} \right) ((Z^{\top} Z)^{-1} A (Z^{\top} Z)^{-1} B)_{kj} + \sum_{k,l} Z_{ik} \left( \frac{\partial}{\partial Z_{ij}}  ((Z^{\top} Z)^{-1})_{kl} \right) (A (Z^{\top} Z)^{-1} B)_{lj} \\ 
						& \quad + \sum_{k,l} (Z (Z^{\top} Z)^{-1} A)_{ik} \left( \frac{\partial}{\partial Z_{ij}} ((Z^{\top} Z)^{-1})_{kl} \right) B_{lj} \\
						=& \sum_k \delta_{jk} ((Z^{\top} Z)^{-1} A (Z^{\top} Z)^{-1} B)_{kj} \\
						& \quad + \sum_{k,l} Z_{ik} ( - ((Z^{\top} Z)^{-1})_{kj} (Z(Z^{\top} Z)^{-1})_{il} - ((Z^{\top} Z)^{-1} Z^{\top})_{ki} ((Z^{\top} Z)^{-1})_{jl} ) (A (Z^{\top} Z)^{-1} B)_{lj} \\ 
						& \quad + \sum_{k,l} (Z (Z^{\top} Z)^{-1} A)_{ik} ( - ((Z^{\top} Z)^{-1})_{kj} (Z(Z^{\top} Z)^{-1})_{il} - ((Z^{\top} Z)^{-1} Z^{\top})_{ki} ((Z^{\top} Z)^{-1})_{jl} ) B_{lj} \\
						=& ((Z^{\top} Z)^{-1} A (Z^{\top} Z)^{-1} B)_{jj} - (Z (Z^{\top} Z)^{-1})_{ij} (Z (Z^{\top} Z)^{-1} A (Z^{\top} Z)^{-1} B)_{ij} \\
						& \quad - (Z (Z^{\top} Z)^{-1} Z^{\top})_{ii} ((Z^{\top} Z)^{-1} A (Z^{\top} Z)^{-1} B)_{jj} - (Z (Z^{\top} Z)^{-1} A (Z^{\top} Z)^{-1})_{ij} (Z (Z^{\top} Z)^{-1} B)_{ij} \\
						& \quad - (Z (Z^{\top} Z)^{-1} A (Z^{\top} Z)^{-1} Z^{\top})_{ii} ((Z^{\top} Z)^{-1} B)_{jj}.
					\end{align*}
					Therefore,
					\begin{align*}
						&\sum_{ij} \frac{\partial}{\partial Z_{ij}} (Z (Z^{\top} Z)^{-1} A (Z^{\top} Z)^{-1} B)_{ij} \\
						=& p {\rm tr} ((Z^{\top} Z)^{-1} A (Z^{\top} Z)^{-1} B) - {\rm tr} ((Z^{\top} Z)^{-1} Z^{\top} \cdot Z (Z^{\top} Z)^{-1} A (Z^{\top} Z)^{-1} B) \\
						& \quad - {\rm tr} (Z (Z^{\top} Z)^{-1} Z^{\top}) {\rm tr}  ((Z^{\top} Z)^{-1} A (Z^{\top} Z)^{-1} B) - {\rm tr}  (Z (Z^{\top} Z)^{-1} A (Z^{\top} Z)^{-1} \cdot B^{\top} (Z^{\top} Z)^{-1} Z^{\top}) \\
						& \quad - {\rm tr}  (Z (Z^{\top} Z)^{-1} A (Z^{\top} Z)^{-1} Z^{\top}) {\rm tr}  ((Z^{\top} Z)^{-1} B) \\
						=& p {\rm tr} ((Z^{\top} Z)^{-1} A (Z^{\top} Z)^{-1} B) - {\rm tr} ((Z^{\top} Z)^{-1} A (Z^{\top} Z)^{-1} B) - q {\rm tr}  ((Z^{\top} Z)^{-1} A (Z^{\top} Z)^{-1} B) \\
						& \quad - {\rm tr}  (A (Z^{\top} Z)^{-1} \cdot B^{\top} (Z^{\top} Z)^{-1}) - {\rm tr}  (A (Z^{\top} Z)^{-1}) {\rm tr}  ((Z^{\top} Z)^{-1} B) \\
						=& (p-q-2) {\rm tr} ((Z^{\top} Z)^{-1} A (Z^{\top} Z)^{-1} B) - {\rm tr}  (A (Z^{\top} Z)^{-1}) {\rm tr}  ((Z^{\top} Z)^{-1} B).
					\end{align*}
				\end{proof}
				
				\section{Expectation formulas}
				
				
				\begin{lemma}\citep{Stein74}\label{lem_stein}
					If $Z \sim {\rm N}_{p,q}(O,I_p,I_q)$ and $g: \mathbb{R}^{p \times q} \to \mathbb{R}^{p \times q}$ is absolutely continuous, then
					\begin{align*}
						{\rm E} [{\rm tr} (Z^{\top} g(Z))] = {\rm E} \left[ \sum_{i,j} \frac{\partial g_{ij}}{\partial Z_{ij}} (Z) \right].
					\end{align*}
				\end{lemma}
				
				
				\begin{lemma}\cite[Theorem~3.3.15 (iv)]{Gupta}\label{lem_gupta3}
					If $d-q \geq 0$, $S \sim W_q(d,\Sigma)$ and $A,B \in \mathbb{R}^{q \times q}$, then
					\begin{align*}
						{\rm E} [{\rm tr} (AS) {\rm tr} (BS)] = d {\rm tr} (A \Sigma B \Sigma) + d {\rm tr} (A^{\top} \Sigma B \Sigma) + d^2 {\rm tr} (A \Sigma) {\rm tr} (B \Sigma).
					\end{align*}
				\end{lemma}
				
				\begin{lemma}\cite[Theorem~3.3.16 (i)]{Gupta}\label{lem_winv}
					If $d-q-1>0$ and $S \sim W_q(d,\Sigma)$, then
					\begin{align*}
						{\rm E} [S^{-1}] = \frac{1}{d-q-1} \Sigma^{-1}.
					\end{align*}
				\end{lemma}
				
				\begin{lemma}\cite[Theorem~3.3.16 (iii)]{Gupta}\label{lem_gupta}
					If $d-q-3>0$, $S \sim W_q(d,\Sigma)$ and $A \in \mathbb{R}^{q \times q}$ is positive semidefinite, then
					\begin{align*}
						{\rm E} [S^{-1} A S^{-1}] = \frac{{\rm tr} (\Sigma^{-1} A)}{(d-q)(d-q-1)(d-q-3)}  \Sigma^{-1} + \frac{1}{(d-q)(d-q-3)} \Sigma^{-1} A \Sigma^{-1}.
					\end{align*}
					In particular, when $\Sigma=A=I_q$,
					\begin{align*}
						{\rm E} [S^{-2}] = \frac{d-1}{(d-q)(d-q-1)(d-q-3)} I_q.
					\end{align*}
				\end{lemma}
				
				\begin{lemma}\cite[Theorem~3.3.17 (ii)]{Gupta}\label{lem_gupta0}
					If $d-q-3>0$ and $S \sim W_q(d,\Sigma)$, then
					\begin{align*}
						{\rm E} [{\rm tr} (S^{-1}) S^{-1}] = \frac{d-q-2}{(d-q)(d-q-1)(d-q-3)} {\rm tr} (\Sigma^{-1}) \Sigma^{-1} + \frac{2}{(d-q)(d-q-1)(d-q-3)} \Sigma^{-2}.
					\end{align*}
					In particular, when $\Sigma=I_q$,
					\begin{align*}
						{\rm E} [{\rm tr} (S^{-1}) S^{-1}] = \frac{q(d-q-2)+2}{(d-q)(d-q-1)(d-q-3)} I_q.
					\end{align*}
				\end{lemma}
				
				\begin{lemma}\cite[Theorem~3.3.17 (iii)]{Gupta}\label{lem_gupta1}
					If $d-q-1>0$ and $S \sim W_q(d,\Sigma)$, then
					\begin{align*}
						{\rm E} [{\rm tr} (S^{-1}) S] = \frac{d}{d-q-1} {\rm tr} (\Sigma^{-1}) \Sigma - \frac{2}{d-q-1} I_q.
					\end{align*}
					In particular, when $\Sigma=I_q$,
					\begin{align*}
						{\rm E} [{\rm tr} (S^{-1}) S] = \frac{dq-2}{d-q-1} I_q.
					\end{align*}
				\end{lemma}
				
				\begin{lemma}\cite{Styan}\label{lem_gupta2}
					If $d-q-1>0$, $S \sim W_q(d,\Sigma)$ and $A \in \mathbb{R}^{q \times q}$, then
					\begin{align*}
						{\rm E} [S A S^{-1}] = \frac{1}{d-q-1} (d \Sigma A \Sigma^{-1} - A^{\top} - {\rm tr} (A) I_q).
					\end{align*}
					In particular, when $\Sigma=I_q$ and $A^{\top}=A$,
					\begin{align*}
						{\rm E} [S A S^{-1}] = \frac{1}{d-q-1} ((d-1) A - {\rm tr} (A) I_q).
					\end{align*}
				\end{lemma}
				
				
				\begin{lemma}\label{lem_exp0}
					If $n-p-q-1>0$ and $Z \sim {\rm N}_{p,q}(\bar{Z},I_p,I_q)$, then
					\begin{align}
						&{\rm E} [{\rm tr} ((Z-\bar{Z})^{\top} (Z-\bar{Z})) {\rm tr} ((Z^{\top} Z)^{-1})] \nonumber \\
						=& pq {\rm E} [ {\rm tr} (  (Z^{\top} Z)^{-1}  ) ] - 2 (p-q-2) {\rm E} [ {\rm tr} ( (Z^{\top} Z)^{-2} ) ] +2 {\rm E} [ ({\rm tr} ((Z^{\top} Z)^{-1}) )^2 ]. \label{expectation0}
					\end{align}
				\end{lemma}
				\begin{proof}
					From Lemma~\ref{lem_stein} and Lemma~\ref{lem_diff2},
					\begin{align}
						&{\rm E} [{\rm tr} ((Z-\bar{Z})^{\top} (Z-\bar{Z})) {\rm tr} ((Z^{\top} Z)^{-1})] \nonumber \\
						=& {\rm E} \left[ \sum_{i,j} \frac{\partial}{\partial Z_{ij}} ( (Z-\bar{Z}))_{ij} {\rm tr} ((Z^{\top} Z)^{-1}) ) \right] \nonumber \\
						=& {\rm E} \left[ \sum_{i,j} ( {\rm tr} ((Z^{\top} Z)^{-1}) - 2 ((Z-\bar{Z}))_{ij} (Z(Z^{\top} Z)^{-1} (Z^{\top} Z)^{-1})_{ij} ) \right] \nonumber \\
						=& pq {\rm E} [ {\rm tr} ((Z^{\top} Z)^{-1}) ] -2 {\rm E} [ {\rm tr} ( (Z-\bar{Z})^{\top} Z(Z^{\top} Z)^{-1} (Z^{\top} Z)^{-1} ) ]. \label{exp01}
					\end{align}
					From Lemma~\ref{lem_stein} and Lemma~\ref{lem_diff3},
					\begin{align}
						{\rm E} [ {\rm tr} ( (Z-\bar{Z})^{\top} Z(Z^{\top} Z)^{-1} (Z^{\top} Z)^{-1} ) ] &= {\rm E} \left[ \sum_{i,j} \frac{\partial}{\partial Z_{ij}} (Z(Z^{\top} Z)^{-1} (Z^{\top} Z)^{-1} )_{ij} \right] \nonumber \\
						&= {\rm E} [ (p-q-2) {\rm tr} ((Z^{\top} Z)^{-2}) - ({\rm tr} ((Z^{\top} Z)^{-1}))^2 ]. \label{exp03}
					\end{align}
					Substituting \eqref{exp03} into \eqref{exp01}, we obtain \eqref{expectation0}.
				\end{proof}
				
				\begin{lemma}\label{lem_exp}
					If $n-p-q-1>0$, $Z \sim {\rm N}_{p,q}(\bar{Z},I_p,I_q)$ and $S \sim W_q ( n-p, (1/n) I_q )$, then
					\begin{align}
						&{\rm E} [{\rm tr} ((Z-\bar{Z})^{\top} (Z-\bar{Z}) S^{-1}) {\rm tr} (S (Z^{\top} Z)^{-1})] \nonumber \\
						=& p \frac{(n-p)q-2}{n-p-q-1} {\rm E} [ {\rm tr} (  (Z^{\top} Z)^{-1}  ) ] - 2 \frac{(n-p-1)(p-q-2)+2}{n-p-q-1} {\rm E} [ {\rm tr} ( (Z^{\top} Z)^{-2} ) ] \nonumber \\
						& \quad +2 \frac{n-q-2}{n-p-q-1} {\rm E} [ ({\rm tr} ((Z^{\top} Z)^{-1}) )^2 ]. \label{expectation}
					\end{align}
				\end{lemma}
				\begin{proof}
					From Lemma~\ref{lem_stein} and Lemma~\ref{lem_diff2},
					\begin{align}
						&{\rm E} [{\rm tr} ((Z-\bar{Z})^{\top} (Z-\bar{Z}) S^{-1}) {\rm tr} (S (Z^{\top} Z)^{-1})] \nonumber \\
						=& {\rm E} \left[ \sum_{i,j} \frac{\partial}{\partial Z_{ij}} ( (Z-\bar{Z}) S^{-1})_{ij} {\rm tr} (S (Z^{\top} Z)^{-1}) ) \right] \nonumber \\
						=& {\rm E} \left[ \sum_{i,j} ( (S^{-1})_{jj} {\rm tr} (S (Z^{\top} Z)^{-1}) - 2 ((Z-\bar{Z}) S^{-1})_{ij} (Z(Z^{\top} Z)^{-1} S (Z^{\top} Z)^{-1})_{ij} ) \right] \nonumber \\
						=& p {\rm E} [ {\rm tr} (S^{-1}) {\rm tr} (S (Z^{\top} Z)^{-1}) ] -2 {\rm E} [ {\rm tr} ( S^{-1} (Z-\bar{Z})^{\top} Z(Z^{\top} Z)^{-1} S (Z^{\top} Z)^{-1} ) ]. \label{exp1}
					\end{align}
					Then, from the linearity of expectation and Lemma~\ref{lem_gupta1},
					\begin{align}
						{\rm E} [ {\rm tr} (S^{-1}) {\rm tr} (S (Z^{\top} Z)^{-1}) ] &= {\rm tr} ( {\rm E} [ {\rm tr} (S^{-1}) S \cdot (Z^{\top} Z)^{-1} ] ) \nonumber \\
						&= {\rm tr} ( {\rm E} [ {\rm tr} (S^{-1}) S  ] \cdot {\rm E} [  (Z^{\top} Z)^{-1} ] ) \nonumber \\
						&= \frac{(n-p)q-2}{n-p-q-1} {\rm E} [ {\rm tr} (  (Z^{\top} Z)^{-1}  ) ]. \label{exp2}
					\end{align}
					Also, from Lemma~\ref{lem_stein} and Lemma~\ref{lem_diff3},
					\begin{align}
						&{\rm E} [ {\rm tr} ( S^{-1} (Z-\bar{Z})^{\top} Z(Z^{\top} Z)^{-1} S (Z^{\top} Z)^{-1} ) ] \nonumber \\
						=& {\rm E} [ {\rm tr} ( (Z-\bar{Z})^{\top} Z(Z^{\top} Z)^{-1} S (Z^{\top} Z)^{-1} S^{-1} ) ] \nonumber \\
						=& {\rm E} \left[ \sum_{i,j} \frac{\partial}{\partial Z_{ij}} (Z(Z^{\top} Z)^{-1} S (Z^{\top} Z)^{-1} S^{-1} )_{ij} \right] \nonumber \\
						=& {\rm E} [ (p-q-2) {\rm tr} ((Z^{\top} Z)^{-1} S (Z^{\top} Z)^{-1} S^{-1}) - {\rm tr} ((Z^{\top} Z)^{-1} S) {\rm tr} ((Z^{\top} Z)^{-1} S^{-1}) ]. \label{exp3}
					\end{align}
					Now, from Lemma~\ref{lem_gupta2},
					\begin{align}
						&{\rm E} [ {\rm tr} ((Z^{\top} Z)^{-1} S (Z^{\top} Z)^{-1} S^{-1}) ] \nonumber \\
						=& {\rm E} [ {\rm tr} ( (Z^{\top} Z)^{-1} \cdot {\rm E} [  S (Z^{\top} Z)^{-1} S^{-1} \mid Z ] ) ] \nonumber \\
						=& {\rm E} \left[ {\rm tr} \left( (Z^{\top} Z)^{-1} \cdot \frac{1}{n-p-q-1} ((n-p-1) (Z^{\top} Z)^{-1} - {\rm tr} ((Z^{\top} Z)^{-1})I_q ) \right) \right] \nonumber \\
						=& \frac{n-p-1}{n-p-q-1} {\rm E} [ {\rm tr} ( (Z^{\top} Z)^{-2} ) ] - \frac{1}{n-p-q-1} {\rm E} [ ({\rm tr} ((Z^{\top} Z)^{-1}) )^2 ]. \label{exp4}
					\end{align}
					Also, by putting $\tilde{S} = ((Z^{\top} Z)^{1/2} S ((Z^{\top} Z)^{1/2} \sim W_q(n-p,Z^{\top} Z)$ and using Lemma~\ref{lem_gupta1},
					\begin{align}
						&{\rm E} [ {\rm tr} ((Z^{\top} Z)^{-1} S) {\rm tr} ((Z^{\top} Z)^{-1} S^{-1}) ] \nonumber \\
						=& {\rm E} [ {\rm tr} ((Z^{\top} Z)^{-2} \tilde{S}) {\rm tr} (\tilde{S}^{-1}) ] \nonumber \\
						=& {\rm E} [ {\rm tr} ((Z^{\top} Z)^{-2} \cdot {\rm E} [ {\rm tr} (\tilde{S}^{-1}) \tilde{S} \mid Z ] ) ] \nonumber \\
						=& {\rm E} \left[ {\rm tr} ((Z^{\top} Z)^{-2} \left( \frac{n-p}{n-p-q-1} {\rm tr} ((Z^{\top} Z)^{-1}) (Z^{\top} Z) - \frac{2}{n-p-q-1} I_q \right) \right] \nonumber \\
						=& \frac{n-p}{n-p-q-1} {\rm E} [ ({\rm tr} ((Z^{\top} Z)^{-1}))^2 ] - \frac{2}{n-p-q-1} {\rm E} [ {\rm tr} ((Z^{\top} Z)^{-2}) ]. \label{exp5}
					\end{align}
					Thus, by substituting \eqref{exp4} and \eqref{exp5} into \eqref{exp3},
					\begin{align}
						&{\rm E} [ {\rm tr} ( S^{-1} (Z-\bar{Z})^{\top} Z(Z^{\top} Z)^{-1} S (Z^{\top} Z)^{-1} ) ] \nonumber \\
						=& (p-q-2) \left( \frac{n-p-1}{n-p-q-1} {\rm E} [ {\rm tr} ( (Z^{\top} Z)^{-2} ) ] - \frac{1}{n-p-q-1} {\rm E} [ ({\rm tr} ((Z^{\top} Z)^{-1}) )^2 ] \right) \nonumber \\
						& \quad  - \left( \frac{n-p}{n-p-q-1} {\rm E} [ ({\rm tr} ((Z^{\top} Z)^{-1}))^2 ] - \frac{2}{n-p-q-1} {\rm E} [ {\rm tr} ((Z^{\top} Z)^{-2}) ] \right) \nonumber \\
						=& \frac{(n-p-1)(p-q-2)+2}{n-p-q-1} {\rm E} [ {\rm tr} ( (Z^{\top} Z)^{-2} ) ] - \frac{n-q-2}{n-p-q-1} {\rm E} [ ({\rm tr} ((Z^{\top} Z)^{-1}) )^2 ].	 \label{exp6}
					\end{align}
					Substituting \eqref{exp2} and \eqref{exp6} into \eqref{exp1}, we obtain \eqref{expectation}.
				\end{proof}

\end{document}